\documentclass[a4paper, reqno, 11pt]{amsart}

\usepackage[english]{babel}
\usepackage{amsmath}
\usepackage{amssymb}
\usepackage{enumerate}
\usepackage{ifthen}
\usepackage{bbm}
\usepackage{graphicx,subfigure}  
\usepackage{geometry}
\provideboolean{shownotes} 
\setboolean{shownotes}{true}
\usepackage{color}
\usepackage{dsfont}
\usepackage{dsfont}
\usepackage{hyperref}
\usepackage{setspace}
\usepackage{soul}
\usepackage{xcolor}
\usepackage{mathtools}
\DeclareMathOperator{\Div}{div}
\newcommand{\lsim}{{\;\raise0.3ex\hbox{$<$\kern-0.75em\raise-1.1ex\hbox{$\sim$}}\;}}
\newcommand{\gsim}{{\;\raise0.3ex\hbox{$>$\kern-0.75em\raise-1.1ex\hbox{$\sim$}}\;}}
\newcommand{\hole}[1]{
\ifthenelse{\boolean{shownotes}}%
{\begin{center} \fbox{ \rule {.25cm}{0cm}
\rule[-.1cm]{0cm}{.4cm} \parbox{.85\textwidth}{\begin{center}
\texttt{#1}\end{center}} \rule {.25cm}{0cm}}\end{center}}
{}
}

\newcommand{\curl}{\mathop{\mathrm {curl}}}

\newtheorem{theorem}{Theorem}[section]

\newtheorem{lemma}[theorem]{Lemma}

\newtheorem{definition}[theorem]{Definition}
\allowdisplaybreaks
\theoremstyle{remark}
\newtheorem{remark}[theorem]{Remark}

\numberwithin{equation}{section}

\subjclass[????]{74H20, 74N20, 35Mxx, 35G25, 35Q74}
\begin{document}

\title[\empty]{Asymptotic Limits for strain-gradient viscoelasticity with nonconvex energy}
\author[A. AlNajjar]{Aseel AlNajjar}
\address[Aseel AlNajjar]{\newline
Computer, Electrical and Mathematical Science and Engineering Division 
\newline
King Abdullah University of Science and Technology (KAUST)
\newline
Thuwal 23955-6900,  Saudi Arabia
}
\email{aseel.alnajjar@kaust.edu.sa}
\author[S. Spirito]{Stefano Spirito}
\address[Stefano Spirito]{\newline 
Department of Information Engineering, Computer Science and Mathematics
\newline
University of L'Aquila
\newline
Via Vetoio
\newline
I-67100 Coppito (L'Aquila), 
Italy
}
\email{stefano.spirito@univaq.it}
\author[A.E. Tzavaras]{Athanasios E. Tzavaras}
\address[Athanasios E. Tzavaras]{
\newline
Computer, Electrical and Mathematical Science and Engineering Division 
\newline
King Abdullah University of Science and Technology (KAUST)
\newline 
Thuwal 23955-6900,  Saudi Arabia
}
\email{athanasios.tzavaras@kaust.edu.sa}
\begin{abstract}
We consider the system of viscoelasticity with higher-order gradients and nonconvex energy in several space dimensions. We establish 
the asymptotic limits when the viscosity $\nu\rightarrow 0$ or when the dispersion coefficient $\delta \rightarrow 0$. For the latter problem, it is worth noting that, for the case of two space dimensions, we also establish a rate of convergence. This result bears analogies to a result of  Chemin \cite{jean1996remark} on the rate of convergence of the zero-viscosity limit  for the two-dimensional Navier-Stokes equations with bounded vorticity.  
\end{abstract}
\maketitle
\section{Introduction}
\noindent
We study the asymptotic limits for the following hyperbolic-parabolic system of viscoelasticity with second-order gradients, or as we shall call it, strain-gradient viscoelasticity,
\begin{equation*}
\tag{VE$_{\nu \delta}$}  
\begin{aligned}
\partial_t u - \Div S(F)&=  \nu\Delta u -\delta \nabla \Delta F 
\\
\partial_t F &= \nabla u
\\
\curl F &= 0  \, ,
\end{aligned}
\label{eqn:Visco}
\end{equation*}
\\
on the periodic domain $\mathbb{T}^d$ for $d=2,3$, with initial data at $\{t=0\} \times \mathbb{T}^d$ given by 
$$
u |_{t=0} = u_0, \hspace{0.2cm}  F|_{t=0} = F_0   \hspace{0.2cm} \mbox{with $\curl F_0 = 0$} \, .
$$
The velocity $u = u^{\nu,\delta}$ and deformation gradient $F = F^{\nu,\delta}$  determine the solution 
$(u^{\nu, \delta}, F^{\nu, \delta})$ of the system \eqref{eqn:Visco} and will
depend on the viscosity $\nu > 0$ and surface tension $\delta > 0$ parameters. 
In this work we study the asymptotic limits of the system \eqref{eqn:Visco} as $\nu \to 0$ or as $\delta \to 0$.
We note that \eqref{eqn:Visco} is equivalent to a second order equation
\begin{equation}
	\label{2nd-evolution}
\frac{\partial^2y}{\partial t^2}-\Div S(\nabla y)-\nu \Delta y_t +\delta \Delta^2y =0 \, ,
\end{equation}
describing the evolution of a motion $y:(0,T) \times \mathbb{T}^d \rightarrow \mathbb{R}^d$, where $\mathbb{T}^d$ is the torus in
$d = 2$ or $3$ dimensions and $\nu > 0$, $\delta > 0$ are positive constants. Setting 
 $$
 \begin{aligned}
 u = \partial_t y  \;  &:  \; (0,T) \times \mathbb{T}^d \rightarrow \mathbb{R}^d,
 \\
F = \nabla y \; &: \; (0,T) \times \mathbb{T}^d \rightarrow \mathbb{M}^{d\times d} \, ,
\end{aligned}
$$
one sees that $(u,F)$ solves \eqref{eqn:Visco}. The equivalence of \eqref{eqn:Visco} and \eqref{2nd-evolution} 
is assured,  for data satisfying $\curl F_0 = 0$,  because the constraint  $\curl F = 0$ is an involution propagated from the initial data 
and ensuring that $F$ is a gradient.
For the initial data we assume that  they may depend on the parameters $\nu, \delta$, 
that $\curl F_0^{\nu, \delta} = 0$, and that they have the regularity
\begin{equation}
    \tag{ID$_{1}$}
    u^{\nu, \delta}_0 \in L^2(\mathbb{T}^d),\hspace{0.1cm} F_0^{\nu,\delta} \in L^p(\mathbb{T}^d), 
    \hspace{0.1cm}\nabla F_0^{\nu,\delta} \in L^2 (\mathbb{T}^d) \, .
    \label{eq:ID}
\end{equation}

The Piola-Kirchhoff stress tensor in \eqref{2nd-evolution} has the form
\begin{equation}\label{constrel}
T_R = S(F) + \nu F_t - \delta \Delta F \, , \quad S(F) = \frac{\partial W(F) }{\partial F}
\end{equation}
and depends on the strain, strain-rate and strain-gradient;  the elastic stress $S(F)$ is determined as the gradient of a stored energy function $W(F)$.  
The incorporation of higher-order gradients allows for models of phase transitions for which a crucial mathematical feature is the lack of convexity of the stored energy function.  As a model in mechanics, the constitutive relation \eqref{constrel} has some deficiencies:
it is not frame indifferent \cite{antman98} and it will not satisfy the conservation of angular momentum.  More importantly, we do not account for the constraint $\det F > 0$ that is necessary to interpret $y$ as a mechanical motion. Nevertheless, we will use \eqref{2nd-evolution} as a mathematical 
template to study the rudiments of the asymptotic limits as the viscosity $\nu \to 0$ or the coefficient of the higher order gradients $\delta \to 0$.

The one-dimensional case of \eqref{eqn:Visco} has been extensively studied in connection to the problem of
phase transitions in continuum mechanics. The lack of monotonicity of the stress induces an instability mechanism and the
higher-gradients serve as a mechanism for penalizing the phase boundaries. 
Serrin \cite{serrin1980phase}  employed the Korteweg theory of capillarity \cite{johannes1901korteweg} to derive equilibrium conditions
for liquid-gas phase transitions in a van-der-Waals fluid, while Slemrod \cite{slemrod1983admissibility}
introduced the viscosity-capillarity admissibility criterion for the selection of admissible propagating phase boundaries.  
The reader is referred to \cite{Truski1982, AK1991} for a discussion of the admissibility issues for phase transitions, to \cite{boldrini1986elasticity,hagan1983viscosity,slemrod1983admissibility} for analysis of traveling waves in connection to viscosity-capillarity criterion, and to \cite{hattori1986riemann,shearer1986nonuniqueness} for the solution of the Riemann problem. The time asymptotics of the viscoelasticity system with non-convex potentials are an intricate problem studied in \cite{andrews1982asymptotic} in one-space dimension, and in \cite{rybka1998convergence,hoffmann2000convergence}  for scalar-valued motions in multi-d models.

A related widely studied model in fluid mechanics is the compressible Navier-Stokes-Korteweg system \cite{burtea2022existence, bresch2019navier, antonelli2019compactness, antonelli2022global}. The vanishing capillarity limit has been studied in \cite{bian2014vanishing} using compactness methods. 
The vanishing viscosity and capillarity limits for the Navier-Stokes-Korteweg are studied in \cite{GL2016} for the one-dimensional case.
All these models fit under the auspices of the general theory proposed by Korteweg \cite{johannes1901korteweg}.

The system \eqref{eqn:Visco} is a diffusive-dispersive approximation of the equations of  nonlinear (hyper)-elasticity.
The main objective of this work is to study the asymptotic limits as the dispersion $\delta \rightarrow 0$ or as the diffusion $\nu \rightarrow 0$. 
Regarding existence of solutions, it is known that under the parameter constraint $\delta = O(\nu^\gamma)$ with $\gamma \geq 2$, the system \eqref{eqn:Visco} can be put into the form of an equivalent fully-parabolic system by setting  $\omega = u -\kappa \Div F$. When the parameter $\kappa$ is appropriately selected
\eqref{eqn:Visco} reduces to 
\begin{equation}
    \begin{cases}
    \partial_t \omega - \Div S(F) = (\nu-\kappa) \Delta \omega\\
    \partial_t F -\nabla \omega = \kappa \Delta F\\
    \curl F = 0.
    \end{cases}
    \label{parabolic}
\end{equation}
The selection of a real $\kappa$ in the range $0 < \kappa < \nu$ requires $\delta = A \nu ^{\gamma}, \gamma \geq 2$.
The transformation, proposed in \cite{slemrod1983admissibility} for single dimension and generalized in \cite{koumatos2020existence} for the multidimensional case, requires that $\nabla F \in L^2$ in order to apply the theory of parabolic systems to \eqref{parabolic} and provides existence
 only in the range $\delta = O(\nu^\gamma)$, $\gamma \geq 2$.

We will proceed here by a careful analysis of the dissipative properties of \eqref{eqn:Visco} on the full range of parameters $\nu, \delta$.
We assume that $W$ is semiconvex (see (H3) in Section \ref{sec2}) and derive in Section \ref{AppA}
the main estimates for  \eqref{eqn:Visco}:  (i) a standard energy estimate
 \begin{equation}\label{intro:energy}
    \frac{d}{dt}\int_{\mathbb{T}^d} \frac{1}{2} |u^{\nu,\delta}|^2 + W(F^{\nu,\delta}) +\frac{\delta}{2} |\nabla F^{\nu,\delta}|^2 dx + \int_{\mathbb{T}^d}\nu| \nabla u^{\nu,\delta}|^2 dx =0 \, , 
\end{equation}
and (ii) a more subtle estimate capturing the full dissipation structure of the form
\begin{multline}\label{intro:dissstr}
   \frac{d}{dt} \int_{\mathbb{T}^d}\Bigg( \frac{1}{2}|u^{\nu,\delta}-\frac{\nu}{2} \Div F^{\nu,\delta}|^2 + \frac{\nu^2}{8} |\Div F^{\nu,\delta}|^2 + W(F^{\nu,\delta}) + \frac{\delta}{2}|\nabla F^{\nu,\delta}|^2\Bigg) dx \\+ \frac{\delta\nu}{2} \int_{\mathbb{T}^d} |\Delta F^{\nu,\delta}|^2 dx + \frac{\nu}{2} \int_{\mathbb{T}^d} D^2 \Tilde{W} : (\nabla F^{\nu,\delta}, \nabla F^{\nu,\delta}) + |\nabla u^{\nu,\delta}|^2 dx \leq \frac{K}{2} \int_{\mathbb{T}^d} |\nabla F^{\nu,\delta}|^2 \, .
\end{multline}
where $\Tilde{W}$ is convex. The estimation is based on a transfer of dissipation from momentum to strain gradients that uses a bracket motivated by compensated compactness.
It captures the full dissipative structure  of the problem and is a key ingredient in the analysis of asymptotic limits.
The existence theory requires data with additional regularity 
$F_0^{\nu,\delta}\in H^1(\mathbb{T}^d)$
and leads to global existence of weak solutions of \eqref{eqn:Visco} for the full range of parameters $\nu, \delta$. It is based on 
\eqref{intro:energy}-\eqref{intro:dissstr}
and a standard  Faedo-Galerkin approximation, see Theorem \ref{thm1} and Appendix \ref{AppB}.

In the limit $\delta \rightarrow 0$, one formally obtains the system of viscoelasticity of Kelvin-Voigt type,  a well-studied system for nonconvex energies with existence in  {\it e.g.}  \cite{Rybka1992,friesecke1997implicit, Demoulini00, koumatos2020existence} and references therein. Uniqueness is known for smooth solutions, and strong  solutions in two dimensions \cite{koumatos2020existence}. We prove here a rate of  convergence result,
\begin{equation}
   \|F^\delta (\cdot, t) -\overline{F} (\cdot, t) \|_r\leq  \big( C \delta^{\frac{r}{2}} \big)^{\exp(-C_2 t)}    \, ,
    \label{intro:rate}
\end{equation}
between solutions $(u^\delta,F^\delta)$ of the strain-gradient viscoelasticity system and $(\bar u,\bar F)$ of viscoelasticity of Kelvin-Voigt type,
in some $L^r$-norm, see Theorem \ref{thm3.1}.
The proof follows the approach of Yudovich \cite{yudovich1963non},  as adapted by Chemin \cite{jean1996remark} for the inviscid limit of the two-dimensional Navier-Stokes equations with bounded vorticity and Larios-Lunasin-Titi \cite{LLT2013} for the 2-d Boussinesq system with no heat-conduction.
This approach is based on a stability estimate exploiting  Osgood's lemma \eqref{osgood}  (rather than on Gr\"{o}nwall's lemma)
and employing results from the critical Sobolev embedding $H^1 \hookrightarrow L^p$ for $p < \infty$ valid in two dimensions. 
The rate of convergence estimate \eqref{intro:rate} complements the uniqueness proof of \cite{koumatos2020existence} and reveals the 2-d version of  \eqref{eqn:Visco} as a system where  the Yudovich idea may be applied to yield uniqueness and stability.

 Finally, for completeness, we consider the convergence of the strain gradient viscoelasticity to Kelvin-Voigt viscoelasticity in higher dmensions,
 and also the convergence as $\nu \rightarrow 0$ from strain-gradient viscoelasticity to the dispersive system of strain-gradient elasticity,
 using compactness methods. These results pinpoint the efficiency of the dissipative structure \eqref{intro:dissstr} in handling the
 respective limiting cases, and are achieved by combining compactness methods with the Aubin-Lions-Simon \cite{simon1986compact} 
 time-compactness framework; see  Theorems \ref{thm4.1} and \ref{thm5.1}, respectively.

The organization of the work goes as follows:  
We start in Section \ref{AppA} by presenting the formal derivation of the estimate \eqref{intro:dissstr}
which captures the dissipative structure of the system.
In Section \ref{sec2}, we state the definition of weak solutions and 
a theorem on existence of weak solutions for initial data $(u^{\nu,\delta}, F^{\nu,\delta}) \in L^2(\mathbb{T}^d)\times H^1(\mathbb{T}^d)$,
see Theorem \ref{thm1}. 
The proof of this theorem is based on Faedo-Galerkin approximation and the formal derivation of the estimate \eqref{intro:dissstr} and is
deferred  to Appendix \ref{AppB}. In Section \ref{sec3}, we establish the rate of convergence in the two-dimensional case. 
The limit as $\delta \rightarrow 0$ is studied in Section \ref{sec4} using compactness methods, and it 
provides an existence theory for the limiting system with initial data satisfying \eqref{eq:ID}.
In Section \ref{sec5}, we establish the limit as the viscosity coefficient $\nu \rightarrow 0$ to a system describing dynamics of
nonlinear elastic materials with strain-gradients and  nonconvex energies. Appendix \ref{AppB} contains the proof of Theorem \ref{thm1},
while Appendix \ref{AppC} contains some auxiliary results used in Section \ref{sec3}.

\section{Derivation of the dissipative structure estimate}\label{AppA} 
In this section we establish the formal derivation of the a-priori estimate \eqref{intro:dissstr} that captures the dissipative structure of \eqref{eqn:Visco},
which expressed in indices reads
\begin{equation}
\begin{cases}
\partial_t u_i = \partial_{\alpha} S_{i\alpha}(F) + \nu \Delta u_i - \delta \partial_{\alpha} \Delta F_{i\alpha}\\ 
\partial_{t}F_{i\alpha} = \partial_{\alpha} u_i\\
\partial_\alpha F_{i\beta} - \partial_{\beta} F_{i\alpha}=0
\end{cases}
\label{eqn:indices}
\end{equation}

To ease the presentation, we use the notation $(u^{\nu,\delta}, F^{\nu,\delta})\equiv (u,F)$.  Recall the energy estimate for the system \eqref{eqn:Visco}\\
  \begin{equation}
   \frac{d}{dt} \int_{\mathbb{T}^d} \frac{1}{2} |u|^2 + W(F) +\frac{\delta}{2} |\nabla F|^2 dx + \nu\int_{\mathbb{T}^d}| \nabla u|^2 dx  =0.
   \label{energy-app}
\end{equation}
\\\\
To derive the dissipative structure estimate \eqref{eqn:transfer}, we first multiply $\eqref{eqn:indices}_1$ by $\partial_\beta F_{i\beta}$ and add $ -|\nabla u|^2 = -\big(\partial_\alpha u_i\big)\big(\partial_\alpha u_i\big)$ on both sides. Then
\begin{align*}
   \Big(\partial_\alpha S_{i\alpha}\Big)\Big(\partial_\beta F_{i\beta}\Big)-\Big(\partial_\alpha u_i\Big)\Big(\partial_\alpha u_i\Big) &= \Big(\partial_{\beta}F_{i\beta}\Big)\Big(\partial_t u_i - \nu\Delta  u_i + \delta \partial_\alpha \Delta F_{i\alpha}\Big) - \Big(\partial_\alpha u_i\Big)\Big(\partial_\alpha u_i\Big)\\
   &=\begin{multlined}[t]
   \Big(\partial_\beta F_{i\beta}\Big) \Big(\delta \partial_\alpha \Delta F_{i\alpha}\Big) + \partial_t \Big(u_i \partial_\beta F_{i\beta}\Big)-u_i \Big(\partial_t \partial_\beta F_{i\beta}\Big)  \\ -\Big(\nu\partial_\beta F_{i\beta}\Big)\Delta u_i-\Big(\partial_\alpha u_i\Big)\Big(\partial_\alpha u_i\Big)
    \end{multlined}\\
    &= \begin{multlined}[t]\Big(\partial_\beta F_{i\beta}\Big) \Big(\delta \partial_\alpha \Delta F_{i\alpha}\Big) + \partial_t \Big(u_i \partial_\beta F_{i\beta}\Big)-u_i \Big(\partial_t \partial_\beta F_{i\beta}\Big) \\ -\Big(\nu\partial_\beta F_{i\beta}\Big)\Big(\partial_t \Div F\Big)-\Big(\partial_\alpha u_i\Big)\Big(\partial_\alpha u_i\Big)\end{multlined}\\
    &=\begin{multlined}[t] \partial_t \Big(u_i \partial_\beta F_{i\beta}\Big)-\partial_\alpha \Big(u_i \partial_t F_{i\alpha}\Big)-\nu\partial_t \frac{1}{2}|\partial_\beta F_{i\beta}|^2 \\+ \partial_\alpha \Big(\partial_\beta F_{i\beta}\delta \Delta F_{i\alpha}\Big)-\Big(\delta \partial_\alpha \partial_\beta F_{i\beta}\Big) \Delta F_{i\alpha} .
    \end{multlined}
\end{align*}
We obtain the identity
\begin{multline*}
\partial_t \Big(\frac{\nu}{2}|\partial_\beta F_{i\beta}|^2-u_i \partial_\beta F_{i\beta}\Big)+ \partial_\alpha\Big(u_i \partial_t F_{i\alpha}-\partial_\beta F_{i\beta}\delta \Delta F_{i\alpha}\Big)+\\ \big(\delta \partial_\alpha \partial_\beta F_{i\beta}\big)\Delta F_{i\alpha} + \big(\partial_\alpha S_{i\alpha}\big)\big(\partial_\beta F_{i\beta}\big)-\big(\partial_\alpha u_i)(\partial_\alpha u_i\big)=0,  
\end{multline*}
which is reminiscent of the identities used in compensated compactness. \\
Next, integrating over the torus gives
\begin{equation}
    \frac{d}{dt}\int_{\mathbb{T}^d} \Big(\frac{\nu}{2}|\Div F|^2 - v \Div F\Big)dx + \delta \int_{\mathbb{T}^d}|\Delta F|^2 dx + \int_{\mathbb{T}^d} \Big(\Div S(F) \Div F - |\nabla u|^2\Big) dx =0,
    \label{eqn:E}
    \end{equation}\\
while adding \eqref{energy-app} to a  $\frac{\nu}{2}$ multiple of \eqref{eqn:E} yields
\begin{equation*}
     \frac{d}{dt}\int \Bigg( \frac{1}{2}|u-\frac{\nu}{2} \Div F|^2 + \frac{\nu^2}{8} |\Div F|^2 + W(F) +\frac{\delta}{2} |\nabla F|^2\Bigg) dx + \frac{\delta\nu}{2}\int  |\Delta F|^2 dx  + J = 0,\\
\end{equation*}
where $J$ is given by
\begin{align*}
    J &:= \frac{\nu}{2}\int \Div S(F) \Div F + |\nabla u|^2 \\
    &= \frac{\nu}{2} \int \frac{\partial ^2 W}{\partial F_{i\alpha}\partial F_{k\gamma}}\partial_\beta F_{k\gamma}\partial_\alpha F_{i\alpha} +|\nabla u|^2\\
    &=\frac{\nu}{2} \int \frac{\partial ^2 W}{\partial F_{i\alpha}\partial F_{k\gamma}}\partial_\alpha F_{k\gamma}\partial_\beta F_{i\alpha} +|\nabla u|^2\\
    &= \frac{\nu}{2} \int D^2 W: (\nabla F, \nabla F) + |\nabla u|^2.
\end{align*}
\\
Substituting back and using (H3) and the notation $\tilde W (F) = W(F) + \tfrac{K}{2} |F|^2$ results to the identity capturing the full dissipative structure:
\begin{multline*}
    \int_{\mathbb{T}^d}\bigg( \frac{1}{2}|u-\frac{\nu}{2} \Div F|^2 + \frac{\nu^2}{8} |\Div F|^2 + W(F) + \frac{\delta}{2}|\nabla F|^2\bigg) dx 
    \\
    + \frac{\delta\nu}{2} \int_0^t\int_{\mathbb{T}^d} |\Delta F|^2 dxds + \frac{\nu}{2} \int_0^t\int_{\mathbb{T}^d} D^2 \Tilde{W} : (\nabla F, \nabla F) + |\nabla u|^2 dxds 
    \\
    \leq \frac{K}{2} \int_{\mathbb{T}^d} |\nabla F|^2 
    + \int_{\mathbb{T}^d}\bigg( \frac{1}{2}|u_0-\frac{\nu}{2} \Div F_0|^2 + \frac{\nu^2}{8} |\Div F_0|^2 + W(F_0) + \frac{\delta}{2}|\nabla F_0|^2\bigg) dx.
\end{multline*}

\section{Existence theory for the system of viscoelasticity with strain-gradient dependence}
\label{sec2}
We next present the precise statements of results:
\subsection*{Notation} We will use the customary Lebesgue spaces $L^p(\mathbb{T}^d)$ and Sobolev spaces $W^{k,p}(\mathbb{T}^d)$. The norms are denoted by $\|\cdot\|_p$ and $\|\cdot\|_{W^{k,p}}$, respectively, and we use the standard convention  $H^s(\mathbb{T}^d):=W^{s,2}(\mathbb{T}^d)$. The Banach-valued space $L^p(0,T;X)$ is the classical Bochner space with norm denoted by $\|\cdot\|_{L^{p}(X)}$. We assume, without loss of generality, that $|\mathbb{T}^d| =1$. The notation $x_j \in_{b} X$ will indicate that  the sequence $\{ x_j \}_{j \in \mathbb{N}}$ is a uniformly bounded sequence
 in $X$. Finally, constants depending on fixed parameters like the dimension, domain or fixed exponents will be suppressed and  we will use the notation $A \lsim B$ to mean $A \leq C B$ for some constant $C >0$.

\subsection*{Assumptions on the Stored Energy Function $W(F)$}
Throughout $p \geq 2$, and the stored energy function is assumed to satisfy: 
\begin{enumerate}
    \item [(H1)] $W \in C^3(\mathbb{R}^{d\times d}; \mathbb{R}).$
    \item[(H2)] There exist constants $c > 0$ and $C > 0$ such that 
    \begin{equation*}
        c \big(|F|^p -1\big) \leq W(F) \leq C \big(|F|^p + 1\big),\hspace{0.2cm} p \geq 2.
    \end{equation*}
     \item [(H3)] $\Tilde{W} := W(F) + \frac{K}{2} |F|^2$ is convex for some $K > 0$. \\
\end{enumerate}

\begin{remark} (H2) and (H3) imply that 
\begin{equation}\label{growthderiv}
    \tag{H4}
    |S(F)| \leq C \big(1+|F|^{p-1}\big), \hspace{0.2cm }\text{for some}\hspace{0.2cm} C > 0.    
\end{equation}
Note that an equivalent set of assumptions is (H1), (H2), (H4) and 
\begin{equation}
    \tag{H5}
   \Big(S(F_1)-S(F_2), F_1-F_2\Big) \geq -K |F_1-F_2|^2 \hspace{0.2cm} \text{for} \hspace{0.2cm} K >0,
\end{equation}
where$(A,B)= \text{tr} \hspace{0.1cm}AB^{T}$. A proof of this fact can be found in \cite{friesecke1997implicit}.\\\\
A strengthened assumption will be imposed in some cases, namely that
\begin{equation}
    \tag{H6} 
    \Big(S(F_1)-S(F_2), F_1-F_2\Big) \geq \Big(C\big(|F_1|^{p-2}+|F_2|^{p-2}\big)-K\Big)|F_1-F_2|^2,
\end{equation}
for some $C>0$ and $K > 0$. (H6) implies that 
\begin{equation}
    \tag{H7}
    D^2 \Tilde{W}(F) \geq c |F|^{p-2} \mathbb{I},
\end{equation}
for some $c>0$. We refer to \cite{koumatos2020existence} for the proof.\\
\end{remark}

\begin{remark}
    The condition (H3) is often referred to as semiconvexity in the literature. The equivalent form (H5) is usually referred to  as Andrews-Ball condition, introduced in \cite{andrews1982asymptotic}.
\end{remark}
\begin{definition}[Weak solution]
We say $(u^{\nu,\delta},F^{\nu,\delta})$
is a weak solution of \eqref{eqn:Visco} for the initial data \eqref{eq:ID} if it satisfies the following:
\begin{enumerate}
    \item  The integrability conditions 
    \begin{equation*}
       u^{\nu,\delta} \in C\big(0,T; L^{2}(\mathbb{T}^d)\big) \cap L^2\big(0,T; H^1(\mathbb{T}^d)\big),
    \end{equation*}
    \begin{equation*}
        F^{\nu,\delta} \in C\big(0,T; L^p(\mathbb{T}^d)\cap H^1(\mathbb{T}^d)\big) \cap L^2\big(0,T;  H^1(\mathbb{T}^d)\cap H^2 (\mathbb{T}^d) \big), \hspace{0.2cm} p \geq 2,\hspace{0.1cm} \delta >0
    \end{equation*}
    \begin{equation*}
        y^{\nu,\delta} \in C\big(0,T; L^2(\mathbb{T}^d)\big)\cap L^\infty \big(0,T; H^1(\mathbb{T}^d)\cap H^2(\mathbb{T}^d)\big) \cap H^1(0,T; L^2(\mathbb{T}^d)).
    \end{equation*}
    
    \item The differential equations \eqref{eqn:Visco} are satisfied in the weak form
    \begin{equation*}
    \begin{aligned}
        \iint u^{\nu,\delta} \phi_t dxdt  &- \iint S(F^{\nu,\delta})\nabla \phi dxdt - \nu \iint \nabla u^{\nu,\delta}\nabla \phi dxdt
        \\
        &+ \delta \iint \Delta F^{\nu,\delta}\nabla \phi dxdt + \int u_0^{\nu,\delta}\phi(0,x) dx = 0
        \end{aligned}
    \end{equation*}
     \begin{equation*}
        \iint F^{\nu,\delta}\psi_t dxdt + \iint \nabla u^{\nu,\delta}\psi dxdt -\int F_0^{\nu,\delta}(x)\psi(0,x)dx = 0
    \end{equation*}
    \begin{equation*}
        F^{\nu,\delta}(t,x) = \nabla y(t,x)\hspace{0.2cm} \text{for a.e.}\hspace{0.2cm}(t,x)\in (0,T)\times \mathbb{T}^d.
    \end{equation*}
    for $\psi, \phi \in C^{\infty}\big([0,T); C^{\infty}(\mathbb{T}^d)\big)$.
    \item The energy inequality is satisfied for a.e. $t\in (0,T)$ 
     \begin{equation}
     \begin{aligned}
    \int_{\mathbb{T}^d} \frac{1}{2} |u^{\nu,\delta}|^2 + W(F^{\nu,\delta}) + \frac{\delta}{2} |\nabla F^{\nu,\delta}|^2 dx 
    &+ \int_0^t\int_{\mathbb{T}^d}\nu| \nabla u^{\nu,\delta}|^2 dx ds 
    \\
    &\leq \int_{\mathbb{T}^d}  \frac{1}{2} |u^{\nu,\delta}_0|^2 + W(F^{\nu,\delta}_0) +\frac{\delta}{2} |\nabla F^{\nu,\delta}_0|^2 dx. 
    \end{aligned}
    \label{eqn:energy}
\end{equation}
\item Dissipative Structure. For a.e. $t\in (0,T)$ we have 
\begin{multline}
    \int_{\mathbb{T}^d}\bigg( \frac{1}{2}|u^{\nu,\delta}-\frac{\nu}{2} \Div F^{\nu,\delta}|^2 + \frac{\nu^2}{8} |\Div F^{\nu,\delta}|^2 + W(F^{\nu,\delta}) + \frac{\delta}{2}|\nabla F^{\nu,\delta}|^2\bigg) dx 
    + \frac{\delta\nu}{2} \int_0^t\int_{\mathbb{T}^d} |\Delta F^{\nu,\delta}|^2 dxds \\+ \frac{\nu}{2} \int_0^t\int_{\mathbb{T}^d} D^2 \Tilde{W} : (\nabla F^{\nu,\delta}, \nabla F^{\nu,\delta}) + |\nabla u^{\nu,\delta}|^2 dxds \leq \frac{K}{2} \int_0^t\int_{\mathbb{T}^d} |\nabla F^{\nu,\delta}|^2 +\\+ \int_{\mathbb{T}^d}\bigg( \frac{1}{2}|u^{\nu,\delta}_0-\frac{\nu}{2} \Div F^{\nu,\delta}_0|^2 + \frac{\nu^2}{8} |\Div F^{\nu,\delta}_0|^2 + W(F^{\nu,\delta}_0) + \frac{\delta}{2}|\nabla F^{\nu,\delta}_0|^2\bigg) dx.
   \label{eqn:transfer}
\end{multline}
\end{enumerate}
\label{def}
\end{definition}
Note that we have incorporated the a-priori estimates into Definition \ref{def}. This is indeed established as part of the proof in Appendix \ref{AppB}, and is convenient for the study of the asymptotic limits in Section \ref{sec4} and Section \ref{sec5}. \\\\
The next theorem establishes the global existence of weak solutions for fixed $\nu$ and $\delta$. The proof outlined in Appendix \eqref{AppB} is standard and is based on a Faedo-Galerkin approximation. 
\begin{theorem}[Global Existence of Weak Solutions]\label{thm1}
Assume the stored energy $W$ satisfies (H1)-(H3) for some $p \geq 2$. Let $(u_0^{\nu,\delta},F_0^{\nu,\delta})\in L^2(\mathbb{T}^d) \times L^p(\mathbb{T}^d)\cap H^1 (\mathbb{T}^d)$ . Then there exists a weak solution of \eqref{eqn:Visco} in the sense of Definition \eqref{def}. Moreover, if $W$ additionally satisfies (H6) then we also have that 
\begin{equation}
\nabla |F^{\nu,\delta}|^{\frac{p}{2}} \in L^2\big(0,T; L^2(\mathbb{T}^d)\big).
\label{extrereg}
\end{equation}

\end{theorem}

\section{From Strain-Gradient Viscoelasticity to Viscoelasticity; a rate of Convergence in the 2d case}
\label{sec3}
Here we will fix $\nu =1$ and simplify the notation to $(u^{\nu,\delta}, F^{\nu,\delta})\equiv (u^\delta, F^{\delta})$. For the case with zero surface tension $(\delta =0)$ we have the system 
\begin{equation*}
\tag{V}
\begin{cases}
\partial_t \overline{u} = \Div S(\overline{F}) + \Delta \overline{u}\\
\partial_t \overline{F} = \nabla \overline{u}\\\
\curl \overline{F} = 0 
\end{cases}
\label{eqn:elast}   
\end{equation*}
with initial data
$$\Bar{u}|_{t=0} = \overline{u}_0, \hspace{0.2cm }\overline{F}|_{t=0} = \overline{F}_0.$$
The system \eqref{eqn:elast} of viscoelasticity has been studied in \cite{koumatos2020existence}, where existence and uniqueness in the two-dimensional case have been established.
The next theorem shows that, under some restrictions on the growth of $D^2W$ and assuming more regularity on the initial data, we have a rate of convergence from \eqref{eqn:Visco} to \eqref{eqn:elast} in the two-dimensional case.
\begin{theorem}[Rate of convergence from Strain-Gradient Viscoelasticity to Viscoelasticity] \label{thm3.1} 
Let $d=2$ and fix $ T>0$. Assume that $W$ satisfies (H1)-(H3), $F_0^\delta \, ,  \overline{F}_0 \in H^1(\mathbb{T}^2) \cap L^p(\mathbb{T}^2)$ and $u^\delta, \overline{u} \in L^2(\mathbb{T}^d)$.  Let $(u^{\delta}, F^{\delta})$ be a weak solution of \eqref{eqn:Visco} and $(\overline{u}, \overline{F})$ be a weak solution of \eqref{eqn:elast} such that $u^{\delta}_0 = \overline{u}_0$ and $F^{\delta}_0 = \overline{F}_0$. If $W$ additionally satisfies for some $C > 0$ the bound 
\begin{equation}
    |D^2 W(F^{\delta})| \leq C \big(1+|F^{\delta}|^{(p-2)}\big), 
    \tag{H8}
    \label{H8}
\end{equation}
Then, 
\begin{enumerate}
    \item Let $ 2 \leq p < 4$ and select $1 < r < \min \{2, \frac{2}{p-2}\}$. There exist positive constants $C_1> 0$, depending on the initial data, and $C_2 = C_2(r,p) >0$, depending on $r$ and $p$, such that for $t\in [0,T]$ and $\delta$ satisfying
\begin{equation*}
    \delta \leq \frac{\exp\Big\{\frac{4}{r}(1-\exp(C_2 t))\Big\}}{C_1^{2/r}}
\end{equation*}
we have 
\begin{equation}
    \|F^\delta-\overline{F}\|^r_r\leq \Big(C_1 \delta^{\frac{r}{2}}\Big)^{\exp(-C_2 t)} \exp\Big\{2-2\exp(-C_2 t)\Big\}.     
    \label{rate}
\end{equation}

\item If $W$ satisfies the stronger hypothesis (H6), then \eqref{rate} holds for all $p \geq 2$  and for $1 < r < \min \{2, \frac{p}{p-2}\}$.
\end{enumerate}
\end{theorem}
\begin{proof} We provide  the proof in two different cases. In the first case we impose the condition (H3) of semiconvexity and obtain a bound valid for $ 2 \leq p \leq 4$, and in the second case we impose the strengthened condition (H6) to obtain a result valid for all $p \geq 2$.
\\
\\
\textbf{(1)} Imposing the condition (H3).
\\\\
Let 
\begin{equation*}
   X_u  = u^{\delta} - \overline{u}, \hspace{0.4cm}
     Y_F = F^{\delta} -\overline{F}.
\end{equation*}\\
Then $(X_u, Y_F)$ solves
\begin{equation}
    \begin{cases}
    \partial_t X_u = \Div \big(S(F^{\delta})-S(\overline{F})\big) + \Delta X_u - \delta \nabla (\Delta F^{\delta})\\
    \partial_t Y_F = \nabla X_u\\
    \text{curl}\hspace{0.1cm} Y_F =0
    \end{cases}
    \label{eqn:sve}\\
\end{equation}
with initial data
$$X_u|_{t=0} = 0, \hspace{0.3cm}Y_F|_{t=0} = 0. $$ 
\\
From $\eqref{eqn:sve}_2$ we get that for $1 \leq r < \infty$, $t \in (0,T)$
\begin{equation}
    \int_{\mathbb{T}^2} |F^{\delta}-\overline{F}|^r dx \leq \int_0^t \int_{\mathbb{T}^2} |\nabla u^{\delta} - \nabla \overline{u}|\hspace{0.2cm} |F^{\delta}-\overline{F}|^{r-1}dx ds.
    \label{eqn:star}\\
\end{equation}
Now consider $\eqref{eqn:sve}_1$ 
$$\partial_t X_u -\Delta X_u = \Div \Big(S(F^{\delta})-S(\overline{F})-\delta \Delta F^{\delta}\Big).$$
\\
Applying Lemma \eqref{lp regularity}, we get that for $1<r<\infty$, $t\in (0,T)$
\begin{align*}
    \int_{0}^t \|\nabla X_u\|^r_r = \int_{0}^t \|\nabla u^{\delta} - \nabla \overline{u}\|^r_r
    &\leq C(r,d)\int_0^t \|S(F^{\delta})-S(\overline{F})-\delta \Delta F^{\delta}\|^r_r \\
    &\lsim \int_0^t \|S(F^{\delta})-S(\overline{F})\|_r^r + \delta^r \int_0^t \|\Delta F^{\delta}\|_r^r ,
\end{align*}
so that we have 
\begin{equation*}
 \frac{1}{2} \int_0^t \|\nabla u^{\delta} - \nabla \overline{u}\|_r^r \leq \frac{1}{2}\int_0^t \|S(F^{\delta})-S(\overline{F})\|_r^r +\frac{\delta^r}{2}\int_0^t \|\Delta F^{\delta}\|_r^r = \frac{1}{2}\int_0^t \|S(F^{\delta})-S(\overline{F})\|_r^r + I_\delta.   
\end{equation*}
\\For the term $I_\delta$ we first recall that the dissipative structure estimate \eqref{eqn:transfer} gives 
\begin{equation*}
  \frac{\delta}{2} \int_0^t  \int_{\mathbb{T}^2} |\Delta F^{\delta}|^2 dx \leq \frac{K}{2}\int_0^t\int_{\mathbb{T}^2}|\nabla F^{\delta}|^2 \leq C_1,
\end{equation*}
where the constant $C_1 >0$ depends on the initial data. Applying H\"{o}lder's inequality to the term $I_\delta$ we obtain for $1 < r < 2$
\begin{equation*}
    I_\delta = \frac{\delta ^r}{2} \int_0^t \int |\Delta F^{\delta}|^r dx ds 
    \leq \big(\frac{1}{2}\big)^{1-\frac{2}{r}} C_1 \delta^{\frac{r}{2}}, \hspace{0.3cm} 1 < r < 2.
\end{equation*}
Substituting back we get 
\begin{equation}
    \int_0^t \|\nabla u^{\delta} - \nabla \overline{u}\|_r^r \leq \int_0^t \|S(F^{\delta})-S(\overline{F})\|_r^r + \big(\frac{1}{2}\big)^{1-\frac{2}{r}} C_1\delta^{\frac{r}{2}}.
    \label{eqn:dblstar}
\end{equation}
Now consider \eqref{eqn:star} and apply Young's inequality
\begin{equation*}
    \int |F^{\delta}-\overline{F}| dx \lsim \int_0^t \int |\nabla u^{\delta}- \nabla \overline{u}|^r + \int_0^t \int|F^{\delta}-\overline{F}|^r dxds.
\end{equation*}
Applying \eqref{eqn:dblstar}
$$\int |F^{\delta}-\overline{F}|^r dx \lsim \int_0^t \|S(F^{\delta})-S(\overline{F})\|_r^r + \int_0^t \|F^{\delta}-\overline{F}\|_r^r dxds + \big(\frac{1}{2}\big)^{1-\frac{2}{r}} C_1 \delta^{\frac{r}{2}}.$$
From the assumptions (H4) and (H8) we have 
\begin{equation}
    |S(F^{\delta})-S(\overline{F})| \leq \Bigg(1 + |F^{\delta}|^{p-2}+ |\overline{F}|^{p-2}\Bigg)|F^{\delta}-\overline{F}|.
    \label{eqn:tripstar}
\end{equation}
Using \eqref{eqn:tripstar} we get 
\begin{multline*}
    \int |F^{\delta}-\overline{F}|^r dx \lsim \int_0^t \int \Bigg((1+|F^{\delta}|^{p-2}+|\overline{F}|^{p-2})|F^{\delta}-\overline{F}|\Bigg)^r dxds +  \int_0^t \int |F^{\delta}-\overline{F}|^r dxds + \big(\frac{1}{2}\big)^{1-\frac{2}{r}} C_1 \delta^{\frac{r}{2}}\\ \\
    \lsim \int_0^t \int |F^{\delta}-\overline{F}|^r dx ds + \int_0^t \int |F^{\delta}|^{r(p-2)}|F^{\delta}-\overline{F}|^r dxds \\\\+\int_0^t \int |\overline{F}|^{r(p-2)}|F^{\delta}-\overline{F}|^r dxds + \big(\frac{1}{2}\big)^{1-\frac{2}{r}}C_1\delta^{\frac{r}{2}}  = I_1 + I_2 + I_3 + C_1 \delta^{\frac{r}{2}}.
\end{multline*}
Consider the integral $I_2$. Applying H\"{o}lder's inequality and the Gagliardo-Nirenberg-Sobolev inequalities in Lemma \eqref{gag} 
\begin{align*}
    I_2 &= \int_0^t \int |F^{\delta}|^{r(p-2)} |F^{\delta}-\overline{F}|^r dx ds\\
    &= \int_0^t \int |F^{\delta}|^{r(p-2)} |F^{\delta}-\overline{F}| |F^{\delta}-\overline{F}|^{r-1}dxds\\ 
    &\leq \int_0^t \|F^{\delta}\|_{r^2 q(p-2)}^{r(p-2)} \|F^{\delta}-\overline{F}\|_{\frac{rq}{q-1}} \|F^{\delta}-\overline{F}\|_r^{r-1}\\ 
    &\leq C_2 \frac{q}{q-1} \int_0^t \|F^{\delta}\|^{r(p-2)}_{r^2 q(p-2)} \|F^{\delta}-\overline{F}\|_r^{1-\frac{1}{q}}\|\nabla F^{\delta} - \nabla \overline{F}\|^{\frac{1}{q}} \|F^{\delta}-\overline{F}\|_r^{r-1}\\ 
    &\lsim  \frac{q}{q-1} \int_0^t \|F^{\delta}\|^{r(p-2)}_{r^2 q(p-2)} \|F^{\delta}-\overline{F}\|_r^{r-\frac{1}{q}}\|\nabla F^{\delta} - \nabla \overline{F}\|^{\frac{1}{q}} \\ 
    &\lsim \underbrace{\frac{q}{q-1}}_{\leq 2} q^{\frac{r(p-2)}{2}}\int_0^t \|\nabla F^{\delta}\|^{r(p-2)} \|F^{\delta}-\overline{F}\|^{r-\frac{1}{q}}_r \|\nabla F^{\delta}- \nabla \overline{F}\|^{\frac{1}{q}}\\
    &\lsim \hspace{0.1cm} q \int_0^t \|\nabla F^{\delta}\|_2 ^{r(p-2)} \|F^{\delta}-\overline{F}\|^{r-\frac{1}{q}}_r \|\nabla F^{\delta} - \nabla \overline{F}\|_2^{\frac{1}{q}} \, ,
\end{align*}
 where the constant in the last inequality is $C_2 = C_2 (r,p)$ depending only on $r, p$. Note that the above calculations restrict $ 2\leq p < 4$ and we may choose $1 <r < \min \{2, \frac{2}{p-2}\}$.\\
Since $F^{\delta}, \overline{F} \in L^{\infty}\big(0,T; H^1\big)$ we get
\begin{equation*}
    I_2 \lsim q \int_0^t \|F^{\delta}- \overline{F } \|_r^{r-\frac{1}{q}}. 
    \label{eqn:one}
\end{equation*}
Similarly for $I_3$ 
\begin{equation*}
    I_3 = \int_0^t \int |\overline{F}|^{r(p-2)}|F^{\delta}-\overline{F}|^r dxds \lsim q \int_0^t \|F^{\delta}-\overline{F}\|_r^{r-\frac{1}{q}}.
\end{equation*}
Therefore
\begin{align*}
    \int |F^{\delta}-\overline{F}|^r &\lsim \int_0^t \int |F^{\delta}-\overline{F}|^r dxds + \int_0^t \int |F^{\delta}|^{r(p-2)}|F^{\delta}-\overline{F}|^r dxds \\
    &\qquad 
    +   \int_0^t \int   |\overline{F}|^{r(p-2)}|F^{\delta}-\overline{F}|^r dxds   +C_1\delta^{\frac{r}{2}} \\
    &\lsim \hspace{0.4cm}\int_0^t \int |F^{\delta}-\overline{F}|^r dxds + 2q \int_0^t \int |F^{\delta}-\overline{F}|_r^{r-\frac{1}{q}} + C_1\delta^{\frac{r}{2}} \\ 
    &\lsim \hspace{0.4cm} \Big(\sup_t \|F^{\delta}-\overline{F}\|_r^{\frac{1}{q}} + q\Big)\int_0^t \|F^{\delta}-\overline{F}\|_r^{r-\frac{1}{q}}ds + C_1\delta^{\frac{r}{2}}. 
\end{align*}
As $F^{\delta}, \overline{F} \in L^{\infty}\big(0,T; H^1\big)$, we conclude 
$$\int |F^{\delta}-\overline{F}|^r dx \leq C_2 \hspace{0.1cm} q \int_0^t \|F^{\delta}-\overline{F}\|_r^{r-\frac{1}{q}} + C_1\delta^{\frac{r}{2}}. \\ \\ $$
Set $y(t) = \|F^{\delta}-\overline{F}\|_r^r$, note that $y(t) \in C([0,T])$ and study the differential inequality
\begin{equation}
        y(t)\leq C_2\hspace{0.1cm} q \int_0^t y(s)^{1-\frac{1}{q}} ds + C_1\delta^{\frac{r}{2}}, \hspace{0.3cm}
        y(0)=0,
    \label{ineq:1}
\end{equation}
for $q > 2$. Denote $\eta := C_1\delta^{\frac{r}{2}}$ and let 
\begin{equation}
z(t) = C_2 \hspace{0.1cm}q \int_0^t y(s)^{1-\frac{1}{q}} ds + \eta, \hspace{0.3cm}
z(0)= \eta,
\end{equation}
then $y(t) \leq z(t)$, and so 
\begin{equation*}
    \frac{dz}{dt} = C_2 \hspace{0.1cm}q y(t)^{1-\frac{1}{q}}\\\\
    \leq C_2 \hspace{0.1cm}q z^{1-\frac{1}{q}} \, .
\end{equation*}
We allow $q$ to depend on time and optimize it to be
$$q = 2-\log z(t),$$
so that as long as $z(t) < 1$ we have 
\begin{equation*}
\frac{dz}{dt} \leq C_2(2-\log z)z^{1-\frac{1}{2-\log z}}\\
= C_2(2-\log z)z e^{\frac{-\log z}{2-\log z}}\\
\leq C_2(2-\log z) z \, .
\end{equation*}
Integrating in time for $t\leq T$, we finally arrive at
\begin{equation*}
    z(t) \leq C_2 \int_0^\tau (2-\log z(\tau))z(\tau)d\tau + \eta, 
\end{equation*}
we next apply Osgood's lemma \eqref{osgood} with
\begin{equation*}
    \mathcal{M}(x) =  \log(2-\log x)-\log 2.
\end{equation*}
 to obtain  
\begin{equation*}
    y(t) \leq z(t) \leq \Big(C_1 \delta^{\frac{r}{2}}\Big)^{\exp(-C_2 t)} \exp\Big\{2-2\exp(-C_2 t)\Big\}.
\end{equation*}
\\\\
\textbf{(2)} Imposing the condition (H6).
\\\\
\noindent
Doing the same analysis as before we arrive at
\begin{multline*}
    \int |F^{\delta}-\overline{F}|^r dx
    \lsim \int_0^t \int |F^{\delta}-\overline{F}|^r dx ds + \int_0^t \int |F^{\delta}|^{r(p-2)}|F^{\delta}-\overline{F}|^r dxds \\+ \int_0^t \int |\overline{F}|^{r(p-2)}|F^{\delta}-\overline{F}|^r dxds + \big(\frac{1}{2}\big)^{1-\frac{2}{r}}C_1\delta^{\frac{r}{2}}
    = I_1 + I_2 + I_3 + C_1 \delta^{\frac{r}{2}}.
\end{multline*}
Setting $\Gamma _1 = |F^{\delta}|^{\frac{p}{2}}$ and $\Gamma _2 = |\overline{F}|^{\frac{p}{2}}$, we see that $I_2$ becomes
\begin{equation*}
    I_2 = \int_0^t\int \Gamma_1^{\frac{2r(p-2)}{p}}|F^{\delta}-\overline{F}||F^{\delta}-\overline{F}|^{r-1}.
\end{equation*}
Following the analysis as before, we again apply H\"{o}lder's inequality and the Gagliardo-Nirenberg-Sobolev inequalities in Lemma \eqref{gag} to get
\begin{align*}
    I_2 &\leq \int_0^t \|\Gamma_1\|^{2r(p-2)/p}_{2r^2q(p-2)/p}\|F^\delta -\overline{F}\|_{\frac{rq}{q-1}}\|F^\delta - \overline{F}\|_r^{r-1}\\
    &\leq C_2 \frac{q}{q-1}\int_0^t \|\Gamma_1\|^{2r(p-2)/p}_{2r^2q(p-2)/p} \|F^\delta-\overline{F}\|^{r-\frac{1}{q}}_{r}\|\nabla F^\delta-\nabla \overline{F}\|_2^{\frac{1}{q}}\\
    &\leq  C_2\hspace{0.1cm}q^{\frac{r(p-2)}{p}}\int_0^t \|\nabla \Gamma_1\|_2^{2r(p-2)/p}\||F^{\delta}-\overline{F}\|_r^{r-\frac{1}{q}}\|\nabla F^{\delta}- \nabla \overline{F}\|_2^{\frac{1}{q}}\\
    &\lsim q \int_0^t \|\nabla \Gamma_1\|_2^{2r(p-2)/p}\||F^{\delta}-\overline{F}\|_r^{r-\frac{1}{q}}\|\nabla F^{\delta}- \nabla \overline{F}\|_2^{\frac{1}{q}}
\end{align*}
Fixing $r$ such that $1 < r < \min\{2, \frac{p}{p-2}\}$ we obtain 
\begin{equation}
    I_2 \leq C_2 \hspace{0.1cm}q \int_0^t \|\nabla \Gamma_1\|_2^2 \hspace{0.1cm}\|F^{\delta}-\overline{F}\|_r^{r-\frac{1}{q}}ds.
\end{equation}
Continuing in a similar manner for $I_3$, we eventually arrive at the differential inequality 
\begin{equation*}
    y(t)\leq C_2 \hspace{0.1cm} q \int_0^t a(s) y(s)^{1-\frac{1}{q}} ds + C_1\delta^{\frac{r}{2}} \, ,
\end{equation*}
where $a(s) = 1 + \|\nabla \Gamma_1(s)\|_2^2 + \|\nabla \Gamma_2(s)\|_2^2$. As (H6) is imposed, the existence theorem \eqref{thm1} ensures $\nabla |F^\delta|^{\frac{p}{2}}\in L^2(0,T; L^2(\mathbb{T}^2))$. A similar result for $\nabla |\overline{F}|^{\frac{p}{2}}$ can be found in \cite{koumatos2020existence}. Thus $a(s)\in L^1(0,T)$.\\\\
Proceeding in the analysis as before, letting 
\begin{equation*}
    z(t) = C_2 q \int_0^t a(s)y(s)^{1-\frac{1}{q}} ds + C_1 \delta^{\frac{r}{2}},
\end{equation*}
and choosing $q= 2-\log z(t)$, we arrive at 
\begin{equation*}
    z(t) \leq C_2 \hspace{0.1cm} q \int_0^\tau a(\tau)(2-\log z(\tau))z(\tau)d\tau +  C_1 \delta^{\frac{r}{2}}, 
\end{equation*}
Applying Osgood's lemma \eqref{osgood} as before we obtain the bound
\begin{equation*}
    y(t) \leq z(t) \leq \Big(C_1 \delta^{\frac{r}{2}}\Big)^{\exp(-C_2 t)} \exp\Big\{2-2\exp(-C_2 t)\Big\}.
\end{equation*}
\end{proof}

\section{Limit as $\delta \rightarrow 0$; a compactness argument}
\label{sec4}
The next theorem shows that under some additional assumptions on $\nabla F_0$, solutions of \eqref{eqn:Visco} converge, up to a subsequence, to solutions of \eqref{eqn:elast} as $\delta \rightarrow 0$. This also provides an existence result for \eqref{eqn:Visco} when $F_0 \in L^p(\mathbb{T}^d)$. Note that all the bounds on the initial data are uniform.  The notation $(u^{\nu,\delta}, F^{\nu,\delta})\equiv (u^\delta, F^{\delta})$ will again be used here.
\begin{theorem}\label{thm4.1}
Let $p \geq 2$ and assume that $W$ satisfies (H1)-(H3). Suppose that $\delta > 0$ and, as a sequence in $\delta$,
\begin{equation}
    u^{\delta}_0 \in_b L^2(\mathbb{T}^d),
\end{equation}
\begin{equation}
    F^{\delta}_0 \in_b L^p(\mathbb{T}^d),
\end{equation}
\begin{equation}
    \delta^{\frac{1}{2}-\varepsilon}\nabla F^{\delta}_0 \in_b L^2(\mathbb{T}^d) \, ,
    \label{gradf-bound}
\end{equation}

while
\begin{equation}
    (\overline{u}_0, \overline{F}_0) \in L^2(\mathbb{T}^d)\times L^p(\mathbb{T}^d).
\end{equation}
Moreover, assume that 
\begin{equation}
    u^\delta_0 \rightarrow \overline{u}_0 \hspace{0.2cm}\text{in}\hspace{0.2cm} L^2(\mathbb{T}^d), \hspace{0.3cm}  F^\delta_0 \rightarrow \overline{F}_0 \hspace{0.2cm}\text{in}\hspace{0.2cm} L^p(\mathbb{T}^d).
    \label{ID}
\end{equation}
Let $\{(u^{\delta}, F^{\delta})\}_{\delta}$ be a sequence of weak solutions of \eqref{eqn:Visco}. Then, up to a subsequence, there exists $(\overline{u}, \overline{F})$, solution of \eqref{eqn:elast}, such that 

\begin{equation*}
u^{\delta}\rightarrow \overline{u}\hspace{0.2cm}\text{in}\hspace{0.2cm}  C \big( (0,T) ;  L^2(\mathbb{T}^d)\big),
\end{equation*}
\begin{equation*}
F^{\delta}\rightarrow \overline{F}\hspace{0.2cm}\text{in}\hspace{0.2cm} C ( (0,T) ;  L^2(\mathbb{T}^d)).
\end{equation*}

\end{theorem}
\noindent 
\begin{proof}
From the energy estimate \eqref{eqn:energy} we get 
\begin{equation}
    u^{\delta} \in_b \hspace{0.1cm} L^{\infty}\big(0,T; L^2(\mathbb{T}^d)\big)\cap L^2\big(0,T; H^1(\mathbb{T}^d)\big),
\end{equation}
\begin{equation}
    F^{\delta} \in_b \hspace{0.1cm} L^{\infty}\big(0,T; L^p(\mathbb{T}^d)\big).
    \label{f}
\end{equation}
Moreover, from \eqref{eqn:Visco} we have that 
\begin{equation}
    \partial_t u^\delta \in_b \hspace{0.1cm}L^2\big(0,T; H^{-1}(\mathbb{T}^d)\big),
\end{equation}
and so an application of Aubin-Lions-Simon \cite{simon1986compact} lemma gives us 
\begin{equation}
    u^\delta \rightarrow \overline{u} \hspace{0.2cm} \text{in} \hspace{0.2cm} L^2\big(0,T; L^2(\mathbb{T}^d)\big).
    \label{u}
\end{equation}
Given \eqref{f}, we have from Banach-Alaoglu theorem
\begin{equation}
    F^{\delta} \rightharpoonup \overline{F} \hspace{0.2cm} \text{weak-$\ast$} \hspace{0.2cm} \text{in}\hspace{0.2cm} L^{\infty}\big(0,T;L^p(\mathbb{T}^d)\big).
\end{equation}
For the nonlinear term $S(F^{\delta})$, the growth conditions imply
\begin{equation}
    S(F^{\delta}) \in L^{\infty}\big(0,T;L^{\frac{p}{p-1}}(\mathbb{T}^d)\big),
\end{equation}
and so there exists some $g$ in $L^{\infty}\big(0,T;L^{\frac{p}{p-1}}(\mathbb{T}^d)\big)$ such that 
\begin{equation}
    S(F^{\delta})\rightharpoonup g\hspace{0.2cm} \text{weak-$\ast$}  \hspace{0.2cm} \text{in}\hspace{0.2cm} L^{\infty}\big(0,T;L^{\frac{p}{p-1}}(\mathbb{T}^d)\big).
\end{equation}
Next we show that $g = S(\overline{F})$. Taking the weak limit as $\delta \rightarrow 0$ of \eqref{eqn:Visco} we get 

\begin{equation*}
\tag{W}
\begin{cases}
\partial_t \overline{u} = \Div g + \nabla \overline{u}\\
\partial_t \overline{F} = \nabla \overline{u}\\\
\curl \overline{F} = 0 
\end{cases}
\label{eqn:weak}   
\end{equation*}
and now we compare weak solutions of \eqref{eqn:Visco} to weak solutions of \eqref{eqn:weak}. Let $\phi \in C^{\infty}\big(0,T; C^{\infty}(\mathbb{T}^d)\big)$. 
A weak solution for \eqref{eqn:Visco} satisfies
\begin{multline}
\iint u^{\delta} \phi_t dxds - \iint S(F^\delta)\nabla \phi dxds -\iint\nabla u^\delta \nabla \phi dxds \\+\delta \iint \Delta F^{\delta} \nabla \phi dxds + \int u^{\delta}(t)\phi(t,x)-u^{\delta}(0)\phi(0,x) dx = 0,
    \label{wkvisco}
\end{multline}
and a weak solution of \eqref{eqn:weak} satisfies 
\begin{multline}
\iint \overline{u}\phi_t dxds - \iint g \nabla \phi dxds - \iint \nabla \overline{u}\nabla \phi dxds + \int \overline{u}(t)\phi(t,x) - \overline{u}(0)\phi(0,x)dx= 0.
   \label{wkwk}
\end{multline}
Taking the difference of the weak solutions of \eqref{eqn:Visco} and \eqref{eqn:weak}
\begin{multline*}
    \iint (u^{\delta} - \overline{u}) \phi_t dxds - \iint\big(S(F^\delta)-g\big)\nabla \phi dxds - \iint (\nabla u^{\delta}-\nabla \overline{u})\nabla \phi dxds \\+ \delta \iint \Delta F^{\delta} \nabla \phi dxds + \int \big(u^{\delta}(t)-\overline{u}(t)\big)\phi(t,x) - \big(u^{\delta}(0)-\overline{u}(0)\big)\phi(0,x) dx = 0.
\end{multline*}
A density argument allows us to choose $\phi = y^\delta - \overline{y}$ and to obtain
\begin{multline*}
    \frac{1}{2}\int|F^{\delta}-\overline{F}|^2 dx + \int_0^t \int \big(S(F^\delta)-g\big)(F^{\delta}-\overline{F})dxds 
    \\
    = \int_0^t\int |u^\delta-\overline{u}|^2 dxds + \delta \int_0^t\int \Delta F^{\delta} (F^{\delta}-\overline{F})dxds
    \\
    + \int(u^{\delta}(t)-\overline{u}(t))(y^\delta-\overline{y})dx - \int(u^{\delta}_0-\overline{u}_0)(y^{\delta}(0)-\overline{y}(0)) + \frac{1}{2}\int |F^{\delta}_0-\overline{F}_0|^2 dx,
\end{multline*}
adding and subtracting $S(\overline{F})$ gives 
\begin{multline*}
    \frac{1}{2}\int |F^{\delta}-\overline{F}|^2 dx + \int_0^t \int \big(S(F^{\delta})-S(\overline{F})\big)(F^\delta - \overline{F}) dxds \leq \int_0^t \int (g-S(\overline{F}))(F^\delta-\overline{F}) dxds \\+ \int_0^t \int |u^\delta -\overline{u}|^2 dxds + \delta \int_0^t\int \Delta F^{\delta}(F^\delta-\overline{F})dxds\\ + \int \big(u^{\delta}(t)-\overline{u}(t)\big)(y^{\delta}-\overline{y})dx -\int \big(u^{\delta}_0-\overline{u}_0\big)(y^{\delta}(0)-\overline{y}(0)) dx + \frac{1}{2}\int |F_0^\delta-\overline{F}_0|^2.
\end{multline*}
The semiconvexity assumption (H3) implies that 
\begin{equation*}
    \Tilde{S}(F) = S(F) + KF
\end{equation*}
is a monotone map. And so we can write 
\begin{multline*}
    \frac{1}{2}\int |F^{\delta}-\overline{F}|^2 dx + \int_0^t \int \big(\overbrace{\Tilde{S}(F^{\delta})-\Tilde{S}(\overline{F})\big)(F^\delta - \overline{F})}^{\geq 0} dxds \leq \int_0^t \int (g-S(\overline{F}))(F^\delta-\overline{F}) dxds \\+ \int_0^t \int |u^\delta -\overline{u}|^2 dxds + \delta \int_0^t\int \Delta F^{\delta}(F^\delta-\overline{F})dxds + \int \big(u^{\delta}(t)-\overline{u}(t)\big)(y^{\delta}-\overline{y})dx  \\-\int \big(u^{\delta}_0-\overline{u}_0\big)(y^{\delta}(0)-\overline{y}(0)) dx + \frac{1}{2}\int |F_0^\delta-\overline{F}_0|^2 + K \int_0^t\int |F^{\delta}-\overline{F}|^2 dxds. 
\end{multline*}
Then 
\begin{multline}
    \int |F^{\delta}-\overline{F}|^2 dx \lsim \int_0^t \int (g-S(\overline{F}))(F^\delta-\overline{F}) dxds + \int_0^t \int |u^\delta -\overline{u}|^2 dxds  \\+ \delta \int_0^t\int \Delta F^{\delta}(F^\delta-\overline{F})dxds + \int \big(u^{\delta}(t)-\overline{u}(t)\big)(y^{\delta}-\overline{y})dx  -\int \big(u^{\delta}_0-\overline{u}_0\big)(y^{\delta}(0)-\overline{y}(0)) dx \\+ \int |F_0^\delta-\overline{F}_0|^2 +  \int_0^t\int |F^{\delta}-\overline{F}|^2 dxds = J_1+J_2+J_3+J_4+J_5+J_6+J_7.
    \label{calc}
\end{multline}
Now consider the term $ J_3$. The dissipative structure estimate \eqref{eqn:transfer} together with \eqref{gradf-bound} imply, after multiplication by $\delta^{\frac{1}{4}}$,
\begin{equation}
    \delta^{5/4} \int |\Delta F^{\delta}|^2 dx \lsim \delta^{1/4}\int |\nabla F^{\delta}|^2 dx \leq C,
\end{equation}
where the constant $C$ depends on the initial data. Then 
\begin{align*}
  J_3 &= \delta \int_0^t \int \Delta F^{\delta}(F^{\delta}-\overline{F}) dxds \\
  &\leq \int_0^t \Big(\delta^{3/4} \delta^{5/4}\int |\Delta F^{\delta}|^2 dx \Big)^{1/2}\|F^{\delta}-\overline{F}\|_{L^2} \hspace{0.1cm}ds
  \lsim \delta^{3/8}\int_0^t \|F^{\delta}-\overline{F}\|_{L^2}\hspace{0.1cm}ds,
\end{align*}
and so $J_3 \rightarrow 0$ as $\delta \rightarrow 0$. \\\\
Using \eqref{ID} and \eqref{u}, taking the limsup and applying Fatou's lemma where necessary, \eqref{calc} simplifies to 
\begin{equation}
    \limsup_{\delta \rightarrow 0} \int |F^{\delta}-\overline{F}|^2 dx \lsim \limsup_{\delta \rightarrow 0 }(J_1+J_7).
\end{equation}
Consider $J_1$ 
\begin{equation*}
    J_1 := \int_0^t \int (g-S(\overline{F}))(F^{\delta}-\overline{F}) dxds,
\end{equation*}
and note that $(g-S(\overline{F})) \in L^{\frac{p}{p-1}}(\mathbb{T}^d)$. That is 
\begin{equation}
    J_1 := \int_0^t \int (g-S(\overline{F}))(F^{\delta}-\overline{F}) dxds = \int_0^t \langle g - S(\overline{F}), F^{\delta}-\overline{F}\rangle_{L^{\frac{p}{p-1}}, L^p} ds \rightarrow 0,
\end{equation}
since $F^{\delta}\rightharpoonup \overline{F}$ weakly-* in $L^p(\mathbb{T}^d)$. So we are left with  
\begin{equation*}
    \limsup_{\delta \rightarrow 0} \int |F^{\delta}-\overline{F}|^2 dx \lsim \int_0^t \limsup_{\delta \rightarrow 0 }\int |F^{\delta}-\overline{F}|^2 dx ds. 
\end{equation*}
Setting $\xi(t) = \limsup_{\delta \rightarrow 0 }\int |F^{\delta}-\overline{F}|^2 dx $, we get 
\begin{equation*}
\xi(t) \lsim \int_0^t \xi(s) ds, \hspace{0.5cm} \xi(0) = 0,  
\end{equation*}
Gr\"{o}nwall's lemma then implies $\xi(t) \equiv 0$ and thus 
\begin{equation}
    F^{\delta} \rightarrow \overline{F} \hspace{0.2cm} \text{in}\hspace{0.2cm} L^{\infty}\big(0,T;L^2(\mathbb{T}^d)\big),
\end{equation}
and so indeed $g = S(\overline{F})$.
\end{proof}

\section{Strain-Gradient Elasticity; Zero viscosity limit}
\label{sec5}
The next theorem investigates the limit as the viscosity coefficient $\nu \rightarrow 0$. For simplicity we fix $\delta = 1$ and denote $(u^{\nu,\delta}, F^{\nu,\delta})\equiv (u^\nu,F^{\nu}) $. It satisfies the system
\begin{equation*}
\tag{VE$_{\nu,\delta}$}   
\begin{cases}
\partial_t u^{\nu} = \Div S(F^{\nu}) + \nu\Delta u^{\nu} - \nabla \Delta F^{\nu}\\
\partial_t F^{\nu} = \nabla u^{\nu}\\
\curl F^{\nu} =0 
\end{cases}
\label{eqn:delta-nu}
\end{equation*}
with initial data at $\{t=0\} \times \mathbb{T}^d$
$$u^{\nu}|_{t=0} = u_0^{\nu},$$ 
$$F^{\nu}|_{t=0} = F_0^{\nu} = \nabla y_0,$$
and periodic boundary conditions.\\
\noindent
Let $(\hat{u},\hat{F})$ be a solution of the system 
\begin{equation*}
\tag{VE$_{\nu }$}   
\begin{cases}
\partial_t \hat{u} = \Div S(\hat{F}) - \nabla \Delta \hat{F}\\
\partial_t \hat{F} = \nabla \hat{u}\\
\curl \hat{F} =0 
\end{cases}
\label{eqn:delta}
\end{equation*}
with initial data 
$$\hat{u}|_{t=0} = \hat{u}_0,\hspace{0.2cm} \hat{F}|_{t=0} = \hat{F}_0. $$
and periodic boundary conditions.  We will prove:
\begin{theorem}[Zero-viscosity limit] \label{thm5.1}
Let $(u^{\nu}, F^{\nu})$ be a weak solution of \eqref{eqn:delta-nu} with initial data $(u_0^{\nu}, F_0^{\nu})$ which satisfy, as a sequence in $\nu$, 
the uniform bounds
\begin{equation*}
    u_0^{\nu} \in_b L^2(\mathbb{T}^d), \hspace{0.2cm}
   \nabla F_0^{\nu}\in_b L^2(\mathbb{T}^d), \hspace{0.2cm}
   F_0^{\nu} \in_b L^p(\mathbb{T}^d).
\end{equation*}
Then, there exists $(\hat{u}, \hat{F})$ weak solution of \eqref{eqn:delta} such that 
\begin{equation*}
    u^{\nu} \rightharpoonup \hat{u} \hspace{0.2cm} \text{weak-* in} \hspace{0.2cm}L^{\infty}\big(0,T; L^2(\mathbb{T}^d)\big),
\end{equation*}
\begin{equation}\label{convF}
    F^{\nu} \rightarrow \hat{F} \hspace{0.2cm}\text{in} \hspace{0.2cm }C\big(0,T; L^q (\mathbb{T}^d)\big),
\end{equation}
for any $q < p$.
\end{theorem}
%

 Note that the convergence \eqref{convF} can be improved by using available bounds on $\nabla F^\nu$ but we will not pursue that here.
\begin{proof}
	 The energy inequality \eqref{eqn:energy} yields the uniform bounds
    \begin{equation}
        u^{\nu} \hspace{0.1cm} \in_b \hspace{0.1cm} L^{\infty}\big(0,T; L^2(\mathbb{T}^d)\big),
    \end{equation}
    \begin{equation}
        F^{\nu}\hspace{0.1cm} \in_b \hspace{0.1cm} L^{\infty}\big(0,T; L^p(\mathbb{T}^d)\big)\cap L^{\infty}\big(0,T; H^1(\mathbb{T}^d)\big).
        \label{i}
    \end{equation}
By the Banach-Alaoglu theorem, there exist subsequences $\big\{u^{\nu}\big\}_{\nu}$ and $\big\{F^{\nu}\big\}_{\nu}$ such that
\begin{equation}
    u^{\nu} \rightharpoonup \hat{u} \hspace{0.2cm} \text{weak-* in} \hspace{0.2cm}L^{\infty}\big(0,T; L^2(\mathbb{T}^d)\big),
\end{equation}
\begin{equation}
    F^{\nu} \rightharpoonup \hat{F} \hspace{0.2cm} \text{weak-* in} \hspace{0.2cm}L^{\infty}\big(0,T; H^1(\mathbb{T}^d)\big). 
    \label{ii}
\end{equation}
From $\eqref{eqn:delta-nu}_2$ we infer 
\begin{equation}
    \partial_t F^{\nu} \in_b L^{\infty}\big(0,T; H^{-1}(\mathbb{T}^d)\big).
\end{equation}
The Aubin-Lions-Simon lemma \cite{simon1986compact} along with \eqref{i}, \eqref{ii} imply
  \begin{equation}
      F^{\nu}\rightarrow \hat{F} \hspace{0.2cm} \text{in} \hspace{0.2cm} C\big(0,T; L^2(\mathbb{T}^d)\big),
  \end{equation}
and so 
\begin{equation}
    F^\nu \rightarrow \hat{F} \hspace{0.2cm} \text{in} \hspace{0.2cm} L^q((0,T)\times \mathbb{T}^d), \hspace{0.2cm} \text{for} \hspace{0.2cm} q <p.
    \label{iii}
\end{equation}
To deal with the nonlinear term $S(F^{\nu})$, note first that by the growth hypothesis \eqref{growthderiv} we have 
\begin{equation}
    S(F^\nu) \in L^\infty(0,T; L^{\frac{p}{p-1}}(\mathbb{T}^d)).
\end{equation}
For $M >0$ and  $1 \leq r < \infty$, we have
\small
\begin{align*}
    \int_0^t \int  &|S(F^\nu)-S(\hat{F})|^r \, dxds = \int_0^t \int_{\{|F^\nu|> M\}\cup \{|F^\nu|\leq M\}} |S(F^\nu)-S(\hat{F})|^r \hspace{0.1cm} \nabla \phi \hspace{0.1cm} dxds \\
    &\leq \begin{multlined}[t][10.5cm] \int_0^t \int_{\{|F^\nu|> M\}} |S(F^\nu)|^r dxds + \int_0^t \int_{\{|F^\nu| > M\}}|S(\hat{F})|^r dxds 
    \\
    + \int_0^t \int |S(F^\nu)-S(\hat{F})|^r \hspace{0.1cm}  \mathds{1}_{\{|F^\nu| \leq M\}} dxds
    \end{multlined}\\
    &\leq \begin{multlined}[t][10.5cm]   \int_0^t\int_{\{|F^\nu > M|\}}\big(1+|F^\nu|^{r(p-1)}\big) dxds+ \int_0^t\int_{\{|\hat{F} > M|\}}\big(1+|\hat{F}|^{r(p-1)}\big)dxds\\ + \int_0^t \int |S(F^\nu)-S(\hat{F})|^r \hspace{0.1cm}  \mathds{1}_{\{|F^\nu| \leq M\}} dxds  \end{multlined}
    \\
    &= : I_1 + I_2 + I_3 \, .
\end{align*}
\normalsize
Consider first
\begin{equation}
    I_1 = \int_0^t \int_{\{|F^\nu| > M\}} \big(1+ |F^\nu|^{r(p-1)}\big) dxds 
\end{equation}
For $r < \frac{p}{p-1}$, we have $F^\nu \rightarrow \hat{F}$ in $L^{r(p-1)}((0,T)\times \mathbb{T}^d)$ and by Vitali's convergence theorem \cite[Theorem 7.13]{bartle}, the sequence $\{F^\nu\}_{\nu}$ is uniformly integrable in the $L^{r(p-1)}$ norm. That is, for $\varepsilon > 0$ there exists $\delta$ such that for every subset $\mathcal{A}$ with $|\mathcal{A}| < \delta$ we have 
\begin{equation}
    \int_0^t \int_{\mathcal{A}} \big(1 + |F^\nu|^{r(p-1)}\big) \hspace{0.1cm} dxds < \frac{\varepsilon}{3} \, .
\end{equation}
By the Chebychev inequality and \eqref{i}, we can choose $M$ suitably large such that 
\begin{equation}
    \text{meas}\{|F^\nu| \geq M\} \leq \frac{1}{M^p} \int_0^t \int_{\{|F^\nu| > M\}} |F^\nu|^p dxds \leq \frac{C}{M^p}
\end{equation}
and so for $M^p > C/\delta$ we have 
\begin{equation}
   I_1 = \int_0^t \int_{\{|F^\nu| > M\}} \big(1+ |F^\nu|^{r(p-1)}\big) dxds < \frac{\varepsilon}{3}.  
\end{equation}
Similarly, $I_2 < \frac{\varepsilon}{3}$.\\\\
For $I_3$, we first note that \eqref{iii} implies, up to a subsequence, 
\begin{equation*}
    F^\nu \rightarrow \hat{F} \hspace{0.2cm}\text{a.e. in} \hspace{0.2cm} (0,T)\times \mathbb{T}^d,
\end{equation*}
and so, 
\begin{equation*}
    S(F^\nu)\rightarrow S(\hat{F}) \hspace{0.2cm}\text{a.e. in} \hspace{0.2cm} (0,T)\times \mathbb{T}^d \, .
\end{equation*}
Moreover, we have for $r < \frac{p}{p-1}$ \\
\begin{equation*}
    |S(F^\nu)-S(\hat{F})|^r \hspace{0.1cm} \mathds{1}_{\{|F^\nu|\leq M\}} \leq 2^r \big(|S(F^\nu)|^r \mathds{1}_{\{|F^\nu|\leq M\}}+ |S(\hat{F})|^r \hspace{0.1cm}\mathds{1}_{\{|F^\nu| \leq M\}}\big)
    \leq C + |S(\hat{F})|^r \in L^1 (\mathbb{T}^d),
\end{equation*}\\
and so the dominated convergence theorem  implies that for $\varepsilon > 0$, there exists $\nu_0$ such that for $\nu > \nu_0$, 
\begin{equation}
    I_3 = \int_0^t \int |S(F^\nu)-S(\hat{F})|^r \hspace{0.1cm}  \mathds{1}_{\{|F^\nu| \leq M\}} dxds < \frac{\varepsilon}{3}.
\end{equation}
Combining the above we conclude that for $r < \frac{p}{p-1}$,
as   $\nu \rightarrow 0$ we have 
\begin{equation}
    S(F^\nu) \rightarrow S(\hat{F}) \quad \text{in} \hspace{0.1cm} L^r((0,T)\times \mathbb{T}^d) \, . 
\end{equation}
which completes the proof.
\end{proof}

\appendix
\section{Existence via Faedo-Galerkin Approximation} \label{AppB} 
We restate Theorem  \ref{thm1} here for the reader's convenience. 

\begin{theorem}
Assume the stored energy $W$ satisfies (H1)-(H3) for some $p \geq 2$. Let  $(u_0^{\nu,\delta},F_0^{\nu,\delta})\in L^2(\mathbb{T}^d) \times L^p(\mathbb{T}^d)\cap H^1 (\mathbb{T}^d)$. Then there exists a weak solution of \eqref{eqn:Visco} in the sense of Definition \eqref{def}. Moreover if $W$ additionally satisfies (H6) then we also have that $$\nabla |F^{\nu,\delta}|^{\frac{p}{2}} \in L^2\big(0,T; L^2(\mathbb{T}^d)\big).$$
\end{theorem}
\begin{proof}
For simplicity of notation we let $(u, F)$ denote $(u^{\nu,\delta}, F^{\nu,\delta})$ with $\nu$ and $\delta$ fixed. We consider the following Galerkin approximation of \eqref{eqn:Visco}. For $N\in \mathbb{N}$ we have
\begin{equation*}
 \tag{DVE}
    \begin{cases}
        \partial_t u^{ N} = \Div P^N (S(F^{ N})) + \nu \Delta u^{ N} - \delta \nabla \Delta F^{ N}\\
        \partial_t F^{ N} = \nabla u^{ N}\\
        \curl F^{N} = 0
    \end{cases}
    \label{eqn:dve}
\end{equation*}
with initial data 
\begin{equation}
    \tag{ID$_N$}
    u^{N}|_{t=0} = u^{N}_0, \hspace{0.3cm} F^{N}|_{t=0} = F^{ N}_0 = \nabla y^{N}_0,
    \label{idn}
\end{equation}
where $(u^{ N}, F^{ N})$ are defined for $k \in \mathbb{Z}^d$ as 
\begin{align*}
    (u^{\ N}, F^{ N}) &= \Bigg(\sum_{|k|\leq N}\hat{u}^{ N}_k(t) e^{ikx}, \sum_{|k|\leq N}\hat{F}^{ N}_k(t)e^{ikx}\Bigg) = \Bigg(\sum_{|k|\leq N}\hat{u}^{ N}_k e^{ikx}, \sum_{|k|\leq N}\hat{y}^{N}_k(t) \otimes ike^{ikx}\Bigg),
\end{align*}
and we have $\hat{u}^{ N}_k=\text{conj}(\hat{u}^{ N}_{-k}), \hspace{0.2cm} \hat{F}^{ N}_k=\text{conj}(\hat{F}^{ N}_{-k})$ where $\text{conj}(\cdot)$ denotes the complex conjugate. \\
For the nonlinear term $S(F)$ we define the projection operator 
$$P^N : L^2 (\mathbb{T}^d) \rightarrow P^N\big(L^2 (\mathbb{T}^d)\big),$$ 
which projects onto the finite-dimensional space $P^N\big(L^2 (\mathbb{T}^d)\big)$ such that 
$$u \mapsto \sum_{|k|\leq N}\hat{u}_k(t)e^{ikx}, \hspace{0.3cm} F\mapsto \sum_{|k|\leq N}\hat{F}_k(t)e^{ikx},$$
where $\hat{u}_k = \frac{1}{|\mathbb{T}^d|}\int u(t,x)e^{-ixk}$.
\begin{lemma}
    Assume that $W(F^{N} )$ satisfies (H1)-(H3). Let $(u^{ N}_0, F^{ N}_0)= (P^N u_0, P^N F_0)$ with $(u_0, F_0)\in L^{2}(\mathbb{T}^d )\times L^p (\mathbb{T}^d)$ for $p \geq 2$. Then there exists a unique smooth solution of \eqref{eqn:dve}. Moreover, the following energy estimate holds
    \begin{equation}
     \int \Big(\frac{1}{2}|u^{ N}|^2 + W(F^{N})+\frac{\delta}{2}|\nabla F^{N}|^2\Big) dx + \nu \int_0^t\int |\nabla u^{ N}|^2 dx \leq \int \Big(\frac{1}{2} |u_N^0|^2 + W(F_0^N) + \frac{\delta}{2}|\nabla F_0^N|^2\Big) dx.  
        \label{energy:n}
    \end{equation}
\end{lemma}
\begin{proof}
Noting that \eqref{eqn:dve} is equivalent to the system of ordinary differential equations 
\begin{equation*}
    \begin{cases}
    \frac{d}{dt}\hspace{0.1cm} \hat{u}^{N}_k = ik \hspace{0.1cm}\hat{S}^{N}_k (t)-\nu|k|^2 \hspace{0.1cm}\hat{u}^{N}_k(t)+i\delta k^3  \hat{F}^{N}_k(t)\\\\
      \frac{d}{dt} \hspace{0.1cm}\hat{F}^{N}_k (t) = ik \hspace{0.1cm}\hat{u}^{N}_k (t) = i\hat{u}^{N}_k(t)\otimes k
    \end{cases}
\end{equation*}
where 
\begin{equation*}
    \hat{S}^{N}_k (t) = \int S \Bigg(\sum_{|k'|\leq N}\hat{F}^{ N}_{k'} (t) e^{ikx}\Bigg) e^{-ikx}dx.
\end{equation*}
\\
By the growth conditions (H1)-(H3) and the Picard- Lindel\"{o}f theorem, there exists $T^N > 0$ such that $(u^{ N}, F^{ N})$ is the unique smooth solution of \eqref{eqn:dve} on $(0,T^N)\times \mathbb{T}^d$. Next, we derive the energy estimate \eqref{energy:n} which allows us to show continuation for $T^N = T$. Multiplying $\eqref{eqn:dve}_1$ by $u^{N}$ and integrating over $\mathbb{T}^d$ yields 
\begin{equation}
    \frac{d}{dt} \int \frac{1}{2}|u^{N}|^2dx + \int \Bigg(\nabla u^{N}: P^N \big(S(F^{N})\big)+\nu |\nabla u^{ N}|^2 + \delta u^{N} \nabla \Delta F^{ N}\Bigg) dx = 0.
    \label{es1}
\end{equation}
Next, multiply $\eqref{eqn:dve}_2$ by $P^N \big(S(F^{N})\big)$, combine with \eqref{es1} to obtain 
\begin{equation*}
    \frac{d}{dt} \int \Big(\frac{1}{2}|u^{N}|^2 + W(F^{N})\Big) dx + \nu \int |\nabla u^{N}|^2 dx + J = 0,
\end{equation*}
where 
\begin{equation*}
    J := \int \delta u^{N}\nabla \Delta F^{N} = \delta \int u^{N} \nabla \Delta F^{N} dx = \frac{\delta}{2} \frac{d}{dt} \int |\nabla F^{N}|^2 dx,
\end{equation*}
so that we have
\begin{equation}
    \frac{d}{dt} \int \Big(\frac{1}{2} |u^{N}|^2 + W(F^{N}) + \frac{\delta}{2} |\nabla F^{N}|^2 \Big) dx+ \nu \int |\nabla u^{ N}|^2 dx = 0.
    \label{energy:dve}
\end{equation}
Furthermore, (H2) along with Parseval identity imply that 

\begin{equation*}
    \sum_{|k|\leq N} \Bigg(\frac{1}{2} |u^{N}|^2 + |F^{ N}|^2 + \frac{\delta}{2}|\nabla F^{N}|^2\Bigg) dx \leq C,
\end{equation*}
and so the solution can be continued for all $T > 0$.
\end{proof}

In order to show that the approximate solution $(u^{N}, F^{N})$ converges to a solution $(u, F)$ of \eqref{eqn:Visco}, we will derive an additional estimate 
regarding the dissipative structure for the system \eqref{eqn:dve}. We multiply $\eqref{eqn:dve}_1$ by $\Div F^{N}$ and add $-|\nabla u^{N}|^2$ to both sides 
in order to obtain after some re-arrangements
\begin{align*}
	&\Div F^N \cdot \Div P^N S(F^N) - |\nabla u^N|^2  \\
	&= \partial_t \big(u^N \cdot \Div F^N\big) + \Div\big(\frac{\delta}{2}\nabla |\Div F^N|^2-u^N \cdot \nabla u^N\big)- \partial_t \big(\frac{\nu}{2}|\Div F^N|^2\big)-\delta |\Delta F^N|^2,
\end{align*}
which, following the calculations in Section \eqref{AppA}, gives 
\begin{multline*}
    \partial_t \bigg(\frac{\nu}{2}|\Div F^{N}|^2 - \big(u^{N} \cdot \Div F^{N}\big)\bigg)-\Div \bigg(-\frac{\delta}{2}\nabla |\Div F^N|^2 - u^{N} \partial_t F^{N}\bigg) \\+ \delta |\Delta F^{N}|^2 + \Div F^{N} \Div P^N S(F^{N})-|\nabla u^{N}|^2 = 0 \, .
\end{multline*}
Integrating over the torus and in time we get
\begin{multline}
    \int \Big(\frac{\nu}{2}|\Div F^{N}|^2\Big)dx - \int \big(u^{N}\cdot \Div F^{N}\big)dx + \delta \int_0^t \int |\Delta F^{N}|^2 dxds \\ + \int_0^t \int \Div F^{N}\Div P^N S(F^{N})dxds -\int_0^t \int |\nabla u^{N}|^2 dxds 
    \\
    = \int \frac{\nu}{2} |\Div F^{N}_0|^2 - \Big(u^{N}_0 \cdot\Div F^{N}_0\Big)dx
    \label{eqn:E4}
\end{multline}
Next, combine \eqref{energy:n}  with  $\frac{\nu}{2}$ multiple of  \eqref{eqn:E4} to obtain
\begin{multline*}
    \frac{d}{dt} \int_{\mathbb{T}^d} \Bigg(\frac{1}{2}|u^{N}-\frac{\nu}{2} \Div F^{N}|^2 + \frac{\nu^2}{8}|\Div F^{N}|^2 +W(F^N) + \frac{\delta}{2}|\nabla F^{N}|^2\Bigg) dx  \\+ \frac{\delta\nu}{2} \int |\Delta F^{N}|^2 dx + \frac{\nu}{2} \int \Div F^{N}\Div P^N S(F^{N}) + |\nabla u^{N}|^2 =0. 
\end{multline*}
Using the convexity of $\tilde{W}$ we get
\begin{multline*}
    \frac{d}{dt} \int_{\mathbb{T}^d} \Bigg(\frac{1}{2}|u^{N}-\frac{\nu}{2} \Div F^{N}|^2 + \frac{\nu^2}{8}|\Div F^{N}|^2 +W(F^N) + \frac{\delta}{2}|\nabla F^{N}|^2\Bigg) dx  \\+ \frac{\delta\nu}{2} \int |\Delta F^{N}|^2 dx + \frac{\nu}{2} \int \Div F^{N}\Div P^N \tilde{S}(F^{N}) + |\nabla u^{N}|^2 = \frac{K}{2} \int |\Div F|^2,
\end{multline*}
where $\tilde{S} = D \tilde{W}$. As the divergence operator and the projection $P^N$ commute, we can integrate by parts twice and use the involution $\eqref{eqn:dve}_3$ to obtain 
\begin{align}
\int \Div \Big(P^N \tilde S(F^{N})\Big) \Div F^{N} dx &= \int \frac{\partial^2 \Tilde{W}(F^N)}{\partial F^N_{i\alpha} \partial F^N_{k\gamma}} \partial_{\beta} F_{k\gamma}^N \partial_{\alpha}F^N_{i\alpha}\\
&= \int \frac{\partial^2 \Tilde{W}(F^N)}{\partial F^N_{i\alpha} \partial F^N_{k\gamma}} \partial_\beta F^N_{k\gamma} \partial_\beta F^N_{i\alpha} \geq 0
\label{semiconvexity1} 
\end{align}
as $\Tilde{W}$ is convex and so $D^2\Tilde{W} \geq 0$. Integrating in time we get
\begin{multline}
    \int \Big(\frac{1}{2}|u^{N}-\frac{\nu}{2} \Div F^{N}|^2 + \frac{\nu^2}{8}|\nabla F^{N}|^2 +W(F^N) + \frac{\delta}{2}|\nabla F^{N}|^2\Big) dx + \frac{\delta\nu}{2} \int_0^t \int |\Delta F^{N}|^2 dxds \\+ \frac{\nu}{2} \int_0^t\int \Big(\Div F^{N}\Div \tilde{S}(F^{N}) + |\nabla u^{N}|^2\Big)dx ds\leq \frac{K}{2} \int_0^t \int |\nabla F^{N}|^2 dx ds\\+ \int \Big(\frac{1}{2}|u^N_0 - \frac{\nu}{2}\Div F^N_0|^2 + \frac{\nu^2}{8}|\nabla F_0^N|^2 + W(F_0^N)+\frac{\delta}{2}|\nabla F^N_0|^2\Big)dx.
      \label{transfer:n}
\end{multline}
As $W(F^N)$ grows like an $L^p$ norm, the estimates \eqref{energy:n} and \eqref{transfer:n} give the following uniform bounds (for fixed $\nu$ and $\delta$)
\begin{equation}
    F^{N} \in_{b} C\big(0,T; L^p(\mathbb{T}^d)\big),
   \label{bound f1}
\end{equation}
\begin{equation}
    u^N \in_{b} C\big(0,T; L^2(\mathbb{T}^d)\big) \, .
\end{equation}
Moreover, $\eqref{eqn:dve}_2$ gives 
\begin{equation}
    \partial_t F^{N} \in_b L^{\infty}\big(0,T; H^{-1}(\mathbb{T}^d)\big),
\end{equation}
then using  $\eqref{eqn:dve}_1$, the growth condition on $S(F^N)$, and the Sobolev embedding $W^{-1, \frac{p}{p-1}}(\mathbb{T}^d)\supset W^{-1,2}(\mathbb{T}^d)$, we conclude
\begin{equation}
    \partial_t u^{N} \in_b L^2\big(0,T; W^{-1,\frac{p}{p-1}}(\mathbb{T}^d)\big) \, .
\end{equation}
We next apply the Aubin-Lions-Simon lemma \cite{simon1986compact} to both $\{ u^N \}$ and $\{ F^N \}$ to arrive at
\begin{equation}
    u^{N} \rightarrow u \hspace{0.5cm} \text{in} \hspace{0.2cm} L^2\big(0,T; L^2 (\mathbb{T}^d)\big),
    \label{conv-u}
\end{equation}
\begin{equation}
    F^{N} \rightarrow F \hspace{0.5cm} \text{in} \hspace{0.2cm} C\big(0,T; L^2 (\mathbb{T}^d)\big).
    \label{conv-f}
\end{equation}
Moreover, for any $\phi \in H^s (\mathbb{T}^d)$, $s > \frac{d}{2}+1$ such that $\|\phi\|_{H^s} \leq 1$ we have 
\begin{equation*}
   P^N \phi \rightarrow \phi \hspace{0.2cm} \text{in} \hspace{0.2cm} L^q((0,T)\times \mathbb{T}^d) \hspace{0.2cm}\forall \hspace{0.1cm} q  < \infty.
\end{equation*}
For the nonlinear term $S(F^{N})$, 
given $F^{N} \rightarrow F$ $a.e.$ and $S(F^{N}) \in_b L^{\frac{p}{p-1}}(\mathbb{T}^d)$, by the Vitali's convergence theorem \cite[Theorem 7.13]{bartle}, we get 
\begin{equation*}
 S(F^{N}) \rightarrow S(F) \hspace{0.2cm} \text{in} \hspace{0.2cm} L^r \hspace{0.2cm} \text{for} \hspace{0.2cm} r < \frac{p}{p-1},
\end{equation*}

Next, we show that as $N \rightarrow \infty$, \eqref{energy:n} and \eqref{transfer:n} converge, respectively, to the energy estimate \eqref{eqn:energy} and the dissipative structure estimate \eqref{eqn:transfer}.\\
Note that \eqref{energy:n} and \eqref{transfer:n} additionally give the following uniform bounds 
\begin{equation}
  \sqrt{\nu} F^{N} \in_{b} C\big(0,T; H^1(\mathbb{T}^d)\big), \hspace{0.2cm} \sqrt{\delta} F^N  \in_{b} C\big(0,T; H^1(\mathbb{T}^d)\big),
    \label{bound f2}
\end{equation}
\begin{equation}
    \sqrt{\nu \delta} \hspace{0.1cm}\nabla F^N \in_b \hspace{0.1cm} L^2(0,T; H^1(\mathbb{T}^d)),
    \label{conv-gradf}
\end{equation}
\begin{equation}
    \sqrt{\nu} \nabla u^N \in_b \hspace{0.1cm} L^2(0,T; L^2(\mathbb{T}^d)).
    \label{conv-gradu}
\end{equation}
Moreover, 
\begin{equation}
    \partial_t \nabla F^N \in_b \hspace{0.1cm} L^2(0,T; H^{-2}(\mathbb{T})^d).
    \label{time-der-f}
\end{equation}
From \eqref{conv-gradf} and \eqref{time-der-f} and the use of, again, the Aubin-Lions-Simon lemma we get 
\begin{equation}
    \nabla F^N \rightarrow \nabla F \hspace{0.5cm} \text{in} \hspace{0.2cm} L^2(0,T; L^2(\mathbb{T}^d)).
    \label{gradF-L2}
\end{equation}
Furthermore, \eqref{conv-gradu} implies, up to a subsequence, 
\begin{equation}
    \sqrt{\nu} \hspace{0.1cm} \nabla u^N \rightharpoonup \sqrt{\nu}\hspace{0.1cm} \nabla u \hspace{0.5cm} \text{weakly in} \hspace{0.2cm} L^2((0,T)\times \mathbb{T}^d),
\end{equation}
and so we have the lower semicontinuity property 
\begin{equation}
    \int_0^t \int_{\mathbb{T}^d} |\nabla u|^2 dxds \leq \liminf_{N\rightarrow \infty} \int_0^t \int_{\mathbb{T}^d} |\nabla u^N|^2 dxds.
    \label{lsc-u}
\end{equation}
For $W(F^N)$, we note that the growth conditions imply 
\begin{equation}
    \int |W(F^N)-W(F)| dx \leq \Bigg(\int \big(1+|F|^{p-1}+ |F^N|^{p-1}\big)^{\frac{p}{p-1}}dx\Bigg)^{\frac{p-1}{p}} \Bigg(\int |F^N-F|^p\Bigg)^{\frac{1}{p}}
\end{equation}
and so given \eqref{bound f1}, we have that 
\begin{equation}
    W(F^N) \rightarrow W(F) \hspace{0.3cm} \text{in} \hspace{0.1cm} L^1(\mathbb{T}^d).
    \label{conv-W}
\end{equation}
So that as $N \rightarrow \infty$ and given \eqref{conv-u}, \eqref{conv-f}, \eqref{gradF-L2}, \eqref{lsc-u} and \eqref{conv-W}, we obtain 
 \begin{equation*}
    \int_{\mathbb{T}^d} \frac{1}{2} |u|^2 + W(F) + \frac{\delta}{2} |\nabla F|^2 dx + \int_0^t\int_{\mathbb{T}^d}\nu| \nabla u|^2 dx ds 
    \leq \int_{\mathbb{T}^d}  \frac{1}{2} |u_0|^2 + W(F_0) +\frac{\delta}{2} |\nabla F_0|^2 dx. 
\end{equation*}
Next, to show that the dissipative structure estimate \eqref{eqn:transfer} is recovered in the limit $N \rightarrow \infty$, we additionally need to show the lower semicontinuity of the term 
\begin{align*}
    \mathcal{J}[F^N] := \frac{\nu}{2} \int_0^t \int_{\mathbb{T}^d}\Div F^N \Div \Tilde{S}(F^N) dxds &= \frac{\nu}{2} \int_0^t \int_{\mathbb{T}^d} D^2 \Tilde{W}(F^N) :(\nabla F^N, \nabla F^N) dxds\\ &= \frac{\nu}{2}\int_0^t \int_{\mathbb{T}^d} \Xi(F^N, \nabla F^N)  dxds.
\end{align*}
where we invoke \cite[\S 8.2, Theorem 1]{evans}. To that end, we identify $F_{i\alpha}^N$ with $f_j$ and denote 
\begin{equation}
 D^2 \Tilde{W}(F^N) :(\nabla F^N, \nabla F^N) := Q(f):(\nabla f, \nabla f),   
\end{equation}
which in indices reads, 
\begin{equation*}
    \sum_{\substack{i,\alpha \\ j,\beta}} \frac{\partial^2 \Tilde{W}}{\partial F_{i\alpha} \partial F_{j\beta}} \partial_\gamma F_{i\alpha} \partial_\gamma F_{j\beta} = \sum_{m,n} Q_{mn}(f) \big(\partial_\gamma f_m \partial_\gamma f_n\big) 
    = \Xi(f, \nabla f).
\end{equation*}
As $\Tilde{W}$ is convex, $Q$ is positive semidefinite and so $\Xi(f, \nabla f) \geq 0$. Moreover, the mapping $\nabla f \mapsto \Xi(\cdot, \nabla f)$ is convex. By \cite[\S 8.2, Theorem 1]{evans}. we conclude that $\mathcal{J}[F^N]$ is lower semicontinuous and 
\begin{equation}
    \frac{\nu}{2} \int_0^t\int_{\mathbb{T}^d} D^2 \Tilde{W}(F) : (\nabla F, \nabla F) dxds \leq \liminf_{N\rightarrow \infty} \frac{\nu}{2} \int_0^t \int_{\mathbb{T}^d} D^2 \Tilde{W}(F^N) :(\nabla F^N, \nabla F^N) dxds
    \label{lsc-D2W}
\end{equation}
and so as $N\rightarrow \infty$, given \eqref{conv-u}, \eqref{conv-f}, \eqref{gradF-L2}, \eqref{lsc-u}, \eqref{conv-W}, and \eqref{lsc-D2W}, the dissipative structure estimate is recovered.
\\\\Finally, if $W$ satisfies also (H6) then (H7) is implied and we have that 
\begin{equation}
 \int |\nabla |F^{N}|^{\frac{p}{2}}|^2 \lsim \int \frac{\partial^2 \Tilde{W}(F^{N})}{\partial F_{i\beta}^{N} \partial F_{j\gamma}^{N}}\partial_\alpha F_{i\beta}^{N} \partial_\alpha F_{j\gamma}^{N}. 
 \label{semiconvexity2}
\end{equation}
Then \eqref{semiconvexity1} and \eqref{semiconvexity2} imply the following estimate 
\begin{align}
    \int \Big(\frac{1}{2}|u^{N} &-\frac{\nu}{2} \Div F^{N}|^2 + \frac{\nu^2}{8}|\nabla F^{N}|^2 +W(F^N) + \frac{\delta}{2}|\nabla F^{N}|^2\Big) dx 
    \nonumber
    \\
    &+ \frac{\delta\nu}{2} \int_0^t \int |\Delta F^{N}|^2 dxds 
    + \frac{\nu}{2}\int_0^t\int |\nabla |F^N|^{\frac{p}{2}}|^2+ \frac{\nu}{2} \int_0^t\int |\nabla u^{N}|^2dxds 
    \nonumber
    \\
    &\leq \frac{K}{2} \int_0^t \int |\nabla F^{N}|^2 dxds 
    \nonumber
    \\
    &+ \int \Big(\frac{1}{2}|u^N_0 - \frac{\nu}{2}\Div F^N_0|^2 + \frac{\nu^2}{8}|\nabla F_0^N|^2 + W(F_0^N)+\frac{\delta}{2}|\nabla F^N_0|^2\Big)dx.
    \label{grad F}
\end{align}
Now, \eqref{grad F} implies that (for $\nu$ fixed)
\begin{equation}
    \{\sqrt{\nu} \hspace{0.1cm} \nabla |F^{N}|^{\frac{p}{2}}\}_N \subset L^2\big(0,T; L^2(\mathbb{T}^d)\big) \, ,
\end{equation}
which along with the bounds for $\{F^N\}$, \eqref{conv-f}, and use of Vitali's convergence theorem implies  that 
\begin{equation}
   \nabla |F|^{\frac{p}{2}} \in L^2\big(0,T; L^2(\mathbb{T}^d)\big).
\end{equation}
\end{proof}
\section{Auxiliary results}\label{AppC}
In this section we will outline, without proof, some results that were used throughout the paper. In the first lemma, we revisit the well-known Gagliardo-Nirenberg-Sobolev interpolation inequality as well as the critical Sobolev inequality, for the case of two-space dimensions, results that are used in the
analysis of Section \ref{sec3}.
\begin{lemma} 
Let  $f \in H^1(\mathbb{T}^2)$. Then the following inequalities hold: 
\begin{enumerate}
    \item There exists a constant $C > 0$ such that for any $r >1$ and $q \geq 2$ 
    $$\|f\|_{\frac{rq}{q-1}}\leq C_r \frac{q}{q-1}\|f\|_r^{1-\frac{1}{q}}\|f\|_{H^1}^{\frac{1}{q}}.$$
    \item There exists a constant $C>0$ such that for any $q >1$ 
    $$\|f\|_q \leq C \sqrt{q}\|f\|_{H^1}.$$
\end{enumerate}
 \label{gag}
\end{lemma}
The exact constants can be deduced from the references \cite{kozono2008remarks} and \cite{lieb2001analysis}.\\\\
The next lemma gives the maximal $L^p$ regularity for the heat equation on the $d$-dimensional torus. A proof of this classical result can be found in \cite{ladyzhenskaya1968linear} for the case of the entire space and in \cite{koumatos2020existence} for the case of the torus.
\begin{lemma}[]
For a smooth function $G$, let $u$ be a smooth solution of the following initial value problem:
$$\partial_t u - \Delta u = \Div G \hspace{0.2cm} \text{in} \hspace{0.2cm} (0,T)\times \mathbb{T}^d$$ 
$$u|_{t=0} = 0 \hspace{0.2cm} \text{on} \hspace{0.2cm} \{t=0\}\times \mathbb{T}^d.$$
Then, for any $t\in (0,T)$ and for any $1 < r< \infty$ 
$$\int_0^t \|\nabla u\|_r^r \leq C \int_0^t \|G\|_r^r,$$
where $C = C(r,d).$
\label{lp regularity}
\end{lemma}
Finally, we state Osgood's lemma which provides a generalization of Gr\"{o}nwall's inequality. We refer the reader to \cite{chemin1995fluides} for the proof.
\begin{lemma}[Osgood's Lemma.\cite{jean1996remark}]
Let $\rho$ be a positive Borel measurable function, $\gamma$ a locally integrable positive function, and $\mu$ a continuous increasing function. Let us assume that, for some strictly positive number $\alpha$, the function $\rho$ satisfies 
$$\rho(t) \leq \alpha + \int_{t_0}^t \gamma(s) \mu(\rho(s))ds.$$
Then we have 
$$-\mathcal{M}(\rho(t)) + \mathcal{M}(\alpha) \leq \int_{t_0}^t \gamma(s)ds \hspace{0.2cm} \text{with} \hspace{0.2cm} \mathcal{M} = \int_x^1 \frac{dr}{\mu(r)}.$$
\label{osgood}
\end{lemma}
\section*{Acknowledgments}
S. Spirito is partially supported by INdAM-GNAMPA, by the projects PRIN 2020 “Nonlinear evolution PDEs, fluid dynamics and transport equations: theoretical foundations and applications”, PRIN2022 “Classical equations of compressible fluids mechanics: existence and properties of non-classical solutions”, and PRIN2022-PNRR “Some mathematical approaches to climate change and its impacts”. The research of A.E. Tzavaras and A. AlNajjar is partially supported by King Abdullah University
of Science and Technology (KAUST), baseline funds No. BAS/1/1652-01-01.
\bibliographystyle{siam}

\begin{thebibliography}{10}

\bibitem{AK1991}
{\sc R.~Abeyaratne and  J.K.~Knowles}, 
{\em Kinetic relations and the propagation of phase boundaries in solids}, Arch. Rational Mech. Analysis, 114,(1991), 119 - 154.

\bibitem{andrews1982asymptotic}
{\sc G.~Andrews and J.~M. Ball}, {\em Asymptotic behaviour and changes of phase in one-dimensional nonlinear viscoelasticity}, Journal of Differential Equations, 44 (1982), pp.~306--341.

\bibitem{antman98}
{\sc S.~S. Antman}, {\em Physically unacceptable viscous stresses}, Z. Angew.
  Math. Phys., 49 (1998), pp.~980--988.


\bibitem{antonelli2019compactness}
{\sc P.~Antonelli and S.~Spirito}, {\em On the compactness of weak solutions to the {N}avier--{S}tokes--{K}orteweg equations for capillary fluids}, Nonlinear Analysis, 187 (2019), pp.~110--124.

\bibitem{antonelli2022global}
{\sc P.~Antonelli and S.~Spirito}, {\em Global existence of weak solutions to the {N}avier--{S}tokes--{K}orteweg equations}, Annales de l'Institut Henri Poincar{\'e} C, 39 (2022), pp.~171--200.
\bibitem{bartle}
{\sc R.-G. Bartle}, {\em The elements of integration and Lebesgue measure}, John Wiley \& Sons, 2014.
\bibitem{bian2014vanishing}
{\sc D.~Bian, L.~Yao, and C.~Zhu}, {\em Vanishing capillarity limit of the compressible fluid models of {K}orteweg type to the {N}avier--{S}tokes equations}, SIAM Journal on Mathematical Analysis, 46 (2014), pp.~1633--1650.


\bibitem{boldrini1986elasticity}
{\sc J.~L. Boldrini}, {\em Is elasticity the proper asymptotic theory for materials with small viscosity and capillarity?}, Proceedings of the Royal Society of Edinburgh Section A: Mathematics, 103 (1986), pp.~99--127.

\bibitem{bresch2019navier}
{\sc D.~Bresch, M.~Gisclon, and I.~Lacroix-Violet}, {\em On {N}avier--{S}tokes--{K}orteweg and {E}uler--{K}orteweg systems: application to quantum fluids models}, Archive for Rational Mechanics and Analysis, 233 (2019), pp.~975--1025.
\bibitem{burtea2022existence}
{\sc C.~Burtea and B.~Haspot}, {\em Existence of global strong solution for the
  {N}avier--{S}tokes--{K}orteweg system in one dimension for strongly degenerate
  viscosity coefficients}, Pure and Applied Analysis, 4 (2022), pp.~449--485.
\bibitem{chemin1995fluides}
{\sc J.-Y. Chemin}, {\em Fluides parfaits incompressibles}, Soci{\'e}t{\'e} math{\'e}matique de France, 1995.

\bibitem{jean1996remark}
{\sc J.-Y. Chemin},  {\em A remark on the inviscid limit for two-dimensional incompressible fluids}, Communications in partial differential equations, 21 (1996), pp.~1771--1779.

\bibitem{Demoulini00}
S. Demoulini, 
Weak solutions for a class of nonlinear systems of viscoelasticity,
{\it Arch. Rational Mech. Analysis\/} {\bf 155} (2000), 299-334.
\bibitem{evans}
{\sc L.-C. Evans}, {\em Partial differential equations}, vol.~19, American Mathematical Society, 2022.
\bibitem{friesecke1997implicit}
{\sc G.~Friesecke and G.~Dolzmann}, {\em Implicit time discretization and global existence for a quasi-linear evolution equation with nonconvex energy}, SIAM Journal on Mathematical Analysis, 28 (1997), pp.~363--380.
\bibitem{GL2016}
{\sc P.~Germain and P.~LeFloch}, {\em Finite energy method for compressible
  fluids: the {N}avier-{S}tokes-{K}orteweg model}, Comm. Pure Appl. Math., 69
  (2016), pp.~3--61.
\bibitem{hagan1983viscosity}
{\sc R.~Hagan and M.~Slemrod}, {\em The viscosity-capillarity criterion for shocks and phase transitions}, Archive for rational mechanics and analysis, 83 (1983), pp.~333--361.

\bibitem{hattori1986riemann}
{\sc H.~Hattori}, {\em The {R}iemann problem for a {v}an der {W}aals fluid with entropy rate admissibility criterion—isothermal case}, Archive for Rational Mechanics and Analysis, 92 (1986), pp.~247--263.

\bibitem{rybka1998convergence}
{\sc K.-H. Hoffmann and P.~Rybka}, {\em Convergence of solutions to the equation of quasi-static approximation of viscoelasticity with capillarity}, Journal of mathematical analysis and applications, 226 (1998), pp.~61--81.

\bibitem{hoffmann2000convergence}
{\sc K.-H. Hoffmann and P.~Rybka},
{\em On convergence of solutions to the equation of viscoelasticity with capillarity}, Communications in Partial Differential Equations, 25 (2000), pp.~1845--1890.

\bibitem{johannes1901korteweg}
{\sc D.~J. Korteweg}, {\em Sur la forme que prennent les {\'e}quations des mouvements des fluides si l’on tient compte des forces capillaires par des variations de densit{\'e}}, Arch. N{\'e}er. Sci. Exactes S{\'e}r. II, 6 (1901), pp.~1--24.

\bibitem{koumatos2020existence}
{\sc K.~Koumatos, C.~Lattanzio, S.~Spirito, and A.~E. Tzavaras}, {\em Existence and uniqueness for a viscoelastic {K}elvin--{V}oigt model with nonconvex stored energy}, Journal of Hyperbolic Differential Equations, 20 (2023), pp.~433--474.

\bibitem{kozono2008remarks}
{\sc H.~Kozono and H.~Wadade}, {\em Remarks on {G}agliardo--{N}irenberg type inequality with critical {S}obolev space and {BMO}}, Mathematische Zeitschrift, 259 (2008), pp.~935--950.

\bibitem{ladyzhenskaya1968linear}
{\sc O.~Ladyzhenskaya, V.~Solonnikov, and N.~Uraltseva}, {\em Linear and quasilinear parabolic equations of second order}, Translation of Mathematical Monographs, AMS, Rhode Island,  (1968).

\bibitem{LLT2013}
{\sc A. Larios, E. Lunasin and E.S. Titi},  {\em Global well-posedness for the 2D Boussinesq system with anisotropic viscosity and without
heat diffusion}, { J. Differential Equations}, 255 (2013), 2626-2654.

\bibitem{lieb2001analysis}
{\sc E.~H. Lieb and M.~Loss}, {\em Analysis}, vol.~14, American Mathematical Soc., 2001.

\bibitem{Rybka1992} 
{\sc P. Rybka}, Dynamical modelling of phase transitions by means of viscoelasticity in many dimensions, {\it Proc. R. Soc. Edinburgh A}, {\bf 121} (1992), 101-138


\bibitem{serrin1980phase}
{\sc J.~Serrin}, {\em Phase transitions and interfacial layers for {v}an der {W}aals fluids}, in Proceedings of SAFA IV Conference, Recent Methods in Nonlinear Analysis and Applications, Naples, 1980.


\bibitem{shearer1986nonuniqueness}
{\sc M.~Shearer}, {\em Nonuniqueness of admissible solutions of {R}iemann initial value problems for a system of conservation laws of mixed type}, Archive for Rational Mechanics and Analysis, 93 (1986), pp.~45--59.

\bibitem{simon1986compact}
{\sc J.~Simon}, {\em Compact sets in the space {$L^p(0,T ; B)$}}, Annali di Matematica pura ed applicata, 146 (1986), pp.~65--96.


\bibitem{slemrod1983admissibility}
{\sc M.~Slemrod}, {\em Admissibility criteria for propagating phase boundaries in a {v}an der {W}aals fluid}, Archive for Rational Mechanics and Analysis, 81 (1983), pp.~301--315.

\bibitem{Truski1982}
{\sc L.~M. Truskinovskii}, {\em Equilibrium phase boundaries}, Doklady Akademii
  Nauk, 265 (1982), pp.~306--310.
  
\bibitem{yudovich1963non}
{\sc V.~I. Yudovich}, {\em Non-stationary flow of an ideal incompressible liquid}, USSR Computational Mathematics and Mathematical Physics, 3 (1963), pp.~1407--1456.
\end{thebibliography}

\end{document}